\documentclass[a4paper,11pt]{amsart}
\usepackage[utf8]{inputenc}
\usepackage[T1]{fontenc}
\usepackage{lmodern}
\usepackage{amsmath, amsthm, amsfonts, amssymb} 
\usepackage{graphicx}
\usepackage{float}
\usepackage{hyperref}
\usepackage[margin=2.5cm]{geometry}
\usepackage{biblatex}
\usepackage{enumerate}
\usepackage{caption}
\usepackage{orcidlink}
\usepackage[hang,flushmargin]{footmisc}

\newcommand{\R}{\mathbb R}
\newcommand{\N}{\mathbb N}
\newcommand{\Z}{\mathbb Z}
\newcommand{\C}{\mathbb C}
\newcommand{\D}{\mathbb D}

\newcommand{\set}[1]{\left\{#1\right\}}
\newcommand{\abs}[1]{\left|#1\right|}
\newcommand{\norm}[1]{\left\lVert#1\right\rVert}

\DeclareMathOperator{\ima}{Im}
\DeclareMathOperator{\dist}{dist}
\DeclareMathOperator{\Area}{Area}

\newtheorem*{remark}{Remark}
\newtheorem{theorem}{Theorem}
\newtheorem{definition}{Definition}
\newtheorem{proposition}{Proposition}
\newtheorem*{example}{Example}
\newtheorem{lemma}{Lemma}
\newtheorem*{corollary}{Corollary}
\newtheorem*{question}{Question}

\begin{document}

\title{On the geometry of unbounded wandering domains}
\author[B. U\v{c}akar]{Beno U\v{c}akar}
\address{
Faculty of Mathematics and Physics, University of Ljubljana, Jadranska 19, 1000 Ljubljana, Slovenia \newline  \indent
Faculty of Mathematics and Computer Science, University of Barcelona, Gran Via de les Corts Catalanes, 585, 08007 Barcelona, Spain \newline \indent
Institute for Mathematics, Physics and Mechanics, Jadranska 19, 1000 Ljubljana, Slovenia 
}
\address{\normalfont \textsc{ORCID: }\orcidlinkc{0009-0007-2006-4107}}

\email{beno.ucakar@imfm.si}
\thanks{The author was supported by the research program P1-0291 from ARIS, Republic of Slovenia
}
\date{}

\begin{abstract}
    We study the geometry of unbounded wandering domains of entire functions using Arakelian approximation.
    First, we show that, given a uniformly accessible closed set contained in a strip, 
    the connected components of its interior can be realized as escaping or oscillating wandering domains of some entire function.
    The iterates of the function are univalent on these wandering domains, and any unbounded wandering domain remains unbounded under iteration.
    Next, we construct escaping or oscillating wandering domains by lifting bounded wandering domains with prescribed geometry via the exponential map.
    In particular, there exists an entire function that has a half-plane as a wandering domain.
    Finally, we show that, in some precise sense, any simply connected open set can be approximated by escaping or oscillating wandering domains.
    As a direct consequence, we obtain wandering domains whose complements have arbitrarily small areas.
\end{abstract}

\maketitle

\section{Introduction} 

Let $f \colon \C \to \C$ be an entire function and let $f^n$ denote the $n$-th iterate of the function $f$ for $n \ge 0$.
We define the \emph{Fatou set} as the set of all points $z \in \C$ such that the family $(f^n)_{n \ge 1}$ is normal in an open neighbourhood of $z$, and the \emph{Julia set} as its complement.
Intuitively, the Fatou set consists of points with stable dynamics under perturbation, while the dynamics of points in the Julia set are chaotic.
A connected component of the Fatou set is called a \emph{Fatou component} and is either \emph{pre-periodic} if $f^n(\Omega) \cap f^m(\Omega) \neq \emptyset$ for some $n \neq m$,
or a \emph{wandering domain} if $f^n(\Omega) \cap f^m(\Omega) = \emptyset$ whenever $n \neq m$.
See~\cite{HolomorphicDynamics} for an introduction to the theory of complex dynamics.

Our main object of interest in this paper are wandering domains. 
A classical result by Sullivan states that wandering domains do not occur for polynomials, see~\cite{SullivanNoWD}. 
Thus, this phenomenon is exclusive to transcendental entire functions. 
The first example of an entire function with a wandering domain was constructed by Baker in~\cite{BakerFirstWD} using infinite products. 
Since then, many other techniques for constructing wandering domains have been developed,
including approximation theory (\cite{EremenkoLjubichPathologicalExamples}, \cite{GeometryOfWD}, \cite{EremenkoConj}), lift constructions (\cite{HermanLiftWD}), and quasiconformal surgery (\cite{Folding}, \cite{Lazebnik}, \cite{MartiPete_Shishikura}). 

The limit functions on a wandering domain are always constant, see~\cite{Fatou1919}.
Depending on the set of these limit values, a wandering domain is called either \emph{escaping} if $\infty$ is its only limit value, 
\emph{oscillating} if both $\infty$ and at least one finite value occur as limit values, or \emph{orbitally bounded} if all limit values are finite.
Baker's first example of a wandering domain was escaping, while the first oscillating wandering domains were constructed by Eremenko and Lyubich in~\cite{EremenkoLjubichPathologicalExamples}.
Whether orbitally bounded wandering domains actually exist remains a major open problem in the area.

There are many aspects of wandering domains worth studying, but recently attention has turned to the possible geometries wandering domains can take.
More precisely, given an open set $\Omega \subseteq \C$, can this set be a wandering domain of some entire function?
The first to address this question was Boc~Thaler, who, using approximation theory, showed in~\cite{GeometryOfWD} that 
every bounded, connected, and regular open set whose closure has connected complement can be realized as the wandering domain of some entire function.
Recall, that an open set $\Omega$ is called \emph{regular} if $\Omega = \mathring{\overline{\Omega}}$. 
As noted by Boc~Thaler in his paper, regularity is a necessary condition for the set $\Omega$ to be a wandering domain.
These wandering domains can be chosen to be either escaping or oscillating.
Later, in~\cite{EremenkoConj}, Martí-Pete, Rempe, and Waterman developed these ideas further. 
They introduced a change in perspective; instead of starting with an open set and taking its closure, they took a suitable compact set and considered its interior components.
Notice, that the interior of a closed set is a regular open set.
Using this approach, they showed that given a compact set $K \subseteq \C$ with connected complement, there exists an entire function for which the connected components of $\mathring{K}$ are wandering domains.
Thus, more complicated open sets can also be realized as wandering domains, for example, ones whose boundaries form \emph{Lakes of Wada continua}. 

The goal of this paper is to obtain analogous results in the unbounded setting. 
A key tool in our approach is the Arakelian approximation theorem, which enables uniform approximation of holomorphic functions on suitable, possibly unbounded, closed sets called \emph{Arakelian sets}.
We will also make frequent use of a recent characterization of Arakelian sets in~\cite{ArakelianCharacterization} due to Fournodavlos.

Before stating our results, let us briefly discuss the general scheme used to realize simply connected bounded open sets as wandering domains.

Given a compact set $K$ in the right half-plane with connected complement, we construct a family of compact sets $(K_n)_{n \ge 0}$, each with connected complement, such that $K_{n+1} \subseteq \mathring{K}_n$ and $\bigcap_{n \ge 0} K_n = K$.
Next, we choose a sequence of points $(p_n)_{n \ge 1}$ satisfying $p_n \in \mathring{K}_{n-1} \setminus K_{n}$ and $\omega((p_n)_{n \ge 1}) = \partial K$, that is the sequence $(p_n)_{n \ge 1}$ accumulates at every point of $\partial K$.
Using Runge approximation, we then inductively construct a converging sequence of entire functions $(f_n)_{n \ge 1}$ 
such that, on the set $f_n^{n-1}(K_n)$, the function $f_n$ is univalent and acts approximately as a translation to the right, 
while the point $f_n^{n-1}(p_n)$ gets sent to an attracting basin of $f_n$ near the origin. 
Passing to the limit, we obtain an entire function $f$.
Since on the sets $(f^n(K))_{n \ge 0}$ the function $f$ acts approximately as a translation to the right,
the points in $K$ escape uniformly to infinity under iteration, hence $\mathring{K}$ belongs to the Fatou set.
On the other hand, the function $f$ has an attracting basin near the origin. 
Given a point $p_n$, its orbit remains close to that of the compact set $K$ for the first $n-1$ steps of the iteration, 
but on the $n$-th step, it gets sent to the attracting basin.
Consequently, preimages of the attracting basin accumulate on the boundary of $K$, so $\partial K$ has to belong to the Julia set.
The same argument applies to the boundary of any iterate $f^n(K)$, hence the connected components of $\mathring{K}$ are escaping wandering domains.

It is worth emphasizing the importance of injectivity in the above construction.
Firstly, when constructing the function $f_{n+1}$, we want to translate the set $f^n_n(K_{n+1})$ to the right, while sending the point $f^n_n(p_{n+1})$ close to the origin.
For this to be possible, the point $f^n_n(p_{n+1})$ and the set $f^n_n(K_{n+1})$ must be disjoint.
Secondly, the function $f_{n+1}$ is obtained via Runge approximation on the compact set $f_n^n(K_{n+1})$, so the complement of this set has to be connected.
Both assumptions can be guaranteed, provided that the function $f_n$ is injective on the set $f_n^{n-1}(K_n)$.
See Figure~\ref{InjektivnostFail} for an example of how these assumptions can fail if the function $f_n$ is not injective.
As a consequence of the functions $(f_n)_{n \ge 1}$ being injective on their respective sets, 
we can also guarantee that the iterates $f^n$ are univalent on $\mathring{K}$.

\begin{figure}[ht]
    \centering
    \includegraphics[width=0.79\textwidth]{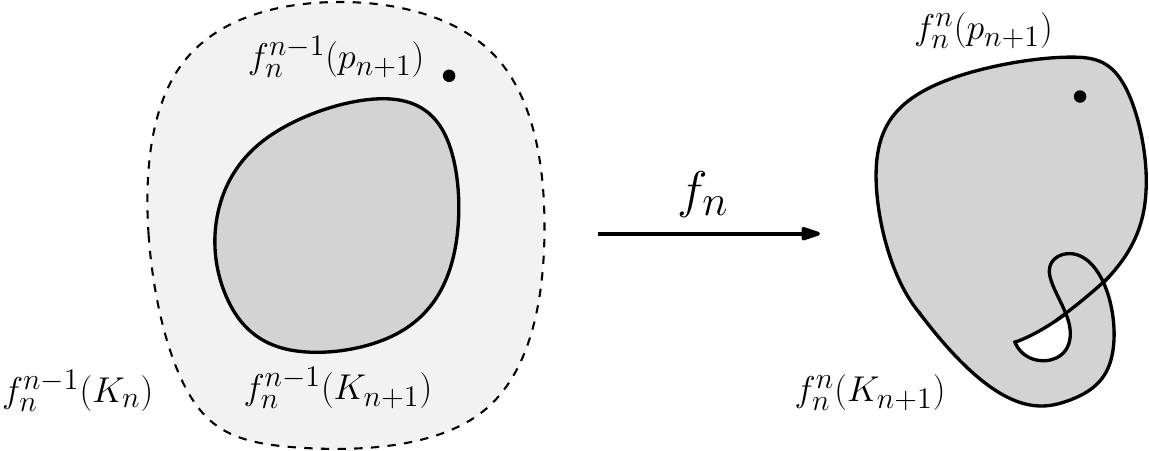}
    \caption{If the function $f_n$ is not injective on $f_n^{n-1}(K_n)$, the assumptions needed for the inductive step might not hold.}
    \label{InjektivnostFail} 
\end{figure}

Given a univalent function defined on a neighbourhood of a compact set $K$, 
one can always ensure that a sufficiently close approximating function is univalent on $K$.
In the unbounded setting, this problem is more delicate.
For a univalent function defined on the neighbourhood of a possibly unbounded closed set $F$,
it is in general not possible to guarantee the injectivity of an approximating function on $F$ by merely approximating well enough. 
Nevertheless, sufficient conditions were established by Martí-Pete, Rempe, and Waterman, see~\cite[Lemma 2.3]{EremenkoConj}.
Note that the assumptions of the following proposition are always satisfied for compact sets.

\begin{proposition}\label{unbounded_injectivity}
    Let $U \subseteq \C$ be an open set and let $\phi \colon U \to \phi(U)$ be a biholomorphic function.
    Suppose that $F \subseteq U$ is a closed set satisfying 
    \[\dist(F, \partial U) > 0 \quad \text{and} \quad \inf_{z \in F} \abs{\phi'(z)} > 0.\]
    Then there exists a constants $\delta > 0$ such that, if $f \colon U \to \C$ is a holomorphic function satisfying $\norm{f-\phi}_U < \delta$,
    the function $f$ is injective on $F$, $f(F) \subseteq \phi(U)$ and $\dist(f(F), \partial \phi(U)) > 0$.
\end{proposition}

\begin{remark}
    The condition $\inf_{z \in F} \abs{\phi'(z)} > 0$ in the above proposition is crucial 
    and is the one that fails when one tries to approximate biholomorphic functions that map unbounded sets to bounded ones.
    For example, let $U$ be the upper half-plane, $F = \set{z \in \C \mid \ima(z) \ge 1}$, and let $\phi \colon U \to \phi(U)$ be either a branch of the logarithm or the Cayley transform $\tau(z) = \frac{z-i}{z+i}$.
    In both cases, the value of $\abs{\phi'}$ degenerates to $0$ along the set $F$.
    Moreover, in both cases we have $\dist(\phi(F), \partial \phi(U)) = 0$, see Figure~\ref{fig_INJ_propade}.
    Similar problems occur if we work with general Riemann maps.
\end{remark}

\begin{figure}[ht]
    \centering
    \includegraphics[width=0.83\textwidth]{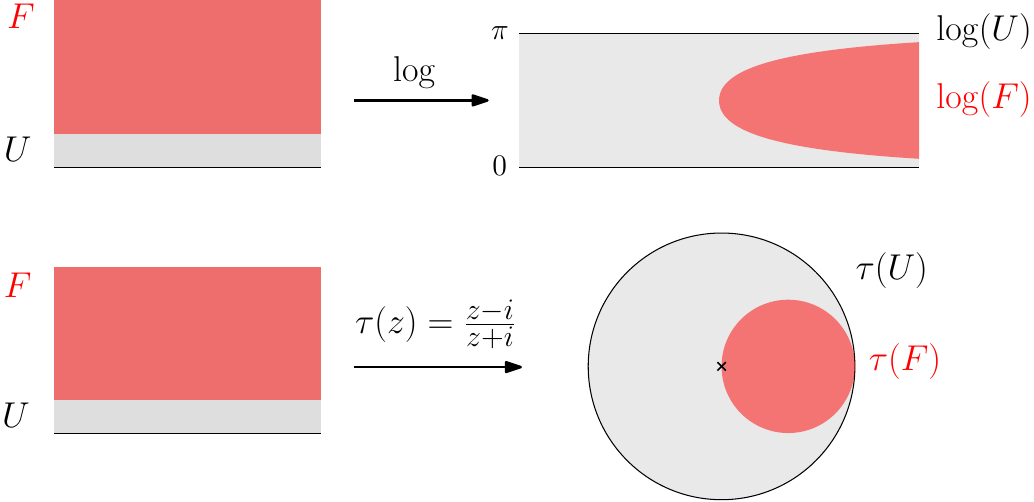}
    \caption{Examples of functions on the upper half-plane that do not satisfy the assumptions of Proposition~\ref{unbounded_injectivity}.}
    \label{fig_INJ_propade} 
\end{figure}

The above remark suggests that, if we want to preserve injectivity under approximation, the best we can do in general is to take $\phi$ to be a linear map.
Constructing wandering domains using linear maps imposes a natural restriction on the closed set $F$ that we can consider; there must exist a linear map $L$ such that the sets $(L^n(F))_{n \ge 0}$ are pairwise disjoint.
The simplest way to ensure this condition is to consider closed sets contained in strips, and for the sake of simplicity, we will consider vertical strips in this paper.
To realize unbounded wandering domains, the closed set $F$ will need to satisfy an additional geometric assumption, which we call \emph{being uniformly accessible}, see Definition~\ref{uniformly_accessible}.
This condition is the analogue, in the unbounded setting, of the requirement that a compact set has to have connected complement in the bounded setting.

\begin{theorem}\label{stripWDescaping}
    Denote the strip $S = [0,1] \times \R \subseteq \C$ and let $F \subseteq S$ be a uniformly accessible closed set such that $\dist(F, \partial S) > 0$.
    Then there exists an entire function $f \in \mathcal{O}(\C)$ for which every connected component of $\mathring{F}$ is an escaping wandering domain.
    Furthermore, the iterates $f^n\vert_{\mathring{F}}$ are univalent, and any unbounded component of $\mathring{F}$ remains unbounded under iteration.
\end{theorem}

\begin{theorem}\label{stripWDoscillating}
    Denote the strip $S = [0,1] \times \R \subseteq \C$ and let $F \subseteq S$ be a uniformly accessible closed set such that $\dist(F, \partial S) > 0$.
    Then there exists an entire function $f \in \mathcal{O}(\C)$ for which every connected component of $\mathring{F}$ is an oscillating wandering domain.
    Furthermore, the iterates $f^n\vert_{\mathring{F}}$ are univalent, and any unbounded component of $\mathring{F}$ remains unbounded under iteration.
\end{theorem}

It was already observed by the authors of~\cite{EremenkoConj} that certain unbounded open sets can be realized as wandering domains. 
In particular, they explain how to modify their construction to obtain an entire function with a half-strip as an oscillating wandering domain.
The above results allow us to realize even more general unbounded open sets as wandering domains of entire functions. 
Furthermore, the iterates of these functions are univalent on their corresponding wandering domains, and any unbounded wandering domain remains unbounded under iteration.

Another way to realize certain unbounded open sets as wandering domains is to lift bounded wandering domains with prescribed geometry via the exponential map.
This approach allows us to construct an entire function with a half-plane as a wandering domain. More generally, we have the following theorem.

\begin{theorem}\label{liftWD}
    Let $\Omega_0$ be a bounded, connected, regular open set containing the origin and whose closure has connected complement. 
    We denote the preimage $\Omega = \exp^{-1}(\Omega_0)$.
    Then there exists an entire function $f \in \mathcal{O}(\C)$ for which $\Omega$ is an escaping or oscillating wandering domain.
    The set $\Omega$ is unbounded and contains a half-plane, while the sets $(f^n(\Omega))_{n \ge 1}$ are bounded. 
    Furthermore, the function $f \vert_\Omega \colon \Omega \to f(\Omega)$ is $\infty$-to-$1$ while the iterates $f^n \vert_{f(\Omega)}$ are univalent.
\end{theorem}

\begin{corollary}
    There exists an entire function that has a half-plane as a wandering domain.
\end{corollary}

Taking the preceding discussion into account, we see that, to realize an unbounded open set as a wandering domain, one must impose significant restrictions on its geometry.
If these geometric restrictions are relaxed, can we prescribe the geometry of a wandering domain \emph{approximately}?
This question is made more precise by the following two theorems. 
In the statements that follow, we adopt the convention that a simply connected set need not itself be connected; rather, each of its connected components is simply connected in the usual sense.

\begin{theorem}\label{approxWD_escaping}
    Let $F \subseteq U \subsetneq \C$, where $F$ is a closed set and $U$ is a simply connected open set whose connected components are all unbounded.     
    Then there exists an entire function $f \in \mathcal{O}(\C)$ such that, for every connected component $F_0$ of $F$,
    the function $f$ has a simply connected escaping wandering domain $\Omega$, such that $F_0 \subseteq \Omega \subseteq U$.
\end{theorem}

\begin{theorem}\label{approxWD_oscillating}
    Let $F \subseteq U \subsetneq \C$, where $F$ is a closed set and $U$ is a simply connected open set whose connected components are all unbounded.     
    Then there exists an entire function $f \in \mathcal{O}(\C)$ such that, for every connected component $F_0$ of $F$,
    the function $f$ has a simply connected oscillating wandering domain $\Omega$, such that $F_0 \subseteq \Omega \subseteq U$.
\end{theorem}

\begin{remark}
    Notice that in the above theorems, we assume $U \neq \C$. 
    This is a technical assumption ensuring that there is enough room for the orbits of the wandering domains to accumulate at infinity.
    It also excludes the pathological case $F = \C$, which, of course, cannot be contained in a wandering domain.
\end{remark}

Let $F$ be a connected closed set and let $F \subseteq U$ be a connected, simply connected, unbounded, open neighbourhood. 
Applying the above theorems yields an entire function $f \in \mathcal{O}(\C)$ with a simply connected wandering domain $\Omega$ such that $F \subseteq \Omega \subseteq U$. 
The wandering domain $\Omega$ is bounded below by $F$ and above by $U$. 
Since the set $F$ is chosen arbitrarily, we can find a wandering domain $\Omega$ which is arbitrarily close to the set $U$.
It is in this sense that we mean that $U$ can be approximated by wandering domains.

When the set $F$ is bounded, this result is somewhat imprecise. 
We only obtain the existence of a simply connected wandering domain $\Omega$ such that $F \subseteq \Omega \subseteq U$, 
but we do not know whether $\Omega$ is bounded, nor how closely its geometry approximates that of $F$. 
In the bounded setting, however, we already know how to realize wandering domains with prescribed geometry. 
Combining Theorems~\ref{approxWD_escaping} and~\ref{approxWD_oscillating} with the results of~\cite{GeometryOfWD} and~\cite{EremenkoConj},
we obtain the following.

\begin{theorem}\label{combination}
    Let $E \subseteq U \subsetneq \C$, where $E$ is an Arakelian set with locally finite connected components and $U$ is a simply connected open set. 
    Then there exists an entire function $f \in \mathcal{O}(\C)$ such that, for every connected component $E_0$ of $E$, one of the following holds:
    \begin{itemize}
        \item If $E_0$ is bounded, then every connected component of $\mathring{E}_0$ is a simply connected wandering domain of $f$.
        \item If $E_0$ is unbounded, then there exists a simply connected wandering domain $\Omega$ of $f$ such that $E_0 \subseteq \Omega \subseteq U$.
    \end{itemize}
    For each connected component $E_0$, the corresponding wandering domains can be chosen to be either escaping or oscillating.
\end{theorem}

As an immediate consequence, we obtain the following corollary,
which removes the unboundedness assumption on the components of $U$ in the statements of Theorems~\ref{approxWD_escaping} and~\ref{approxWD_oscillating}.

\begin{corollary}
    Let $F \subseteq U \subsetneq \C$, where $F$ is a closed set and $U$ is a simply connected open set.     
    Then there exists an entire function $f \in \mathcal{O}(\C)$ such that, for every connected component $F_0$ of $F$,
    the function $f$ has a simply connected wandering domain $\Omega$, such that $F_0 \subseteq \Omega \subseteq U$.
    This wandering domain can be chosen to be either escaping or oscillating.
\end{corollary}

As another immediate consequence, we obtain the existence of wandering domains whose complements have arbitrarily small area. 
This shows that, perhaps somewhat unintuitively, a single wandering domain can occupy most of the complex plane.

\begin{theorem}\label{hornWD}
    For each $\delta > 0$ there exists an entire function $f \in \mathcal{O}(\C)$ with a wandering domain $\Omega$ such that $\Area(\C \setminus \Omega) < \delta$.
\end{theorem}

Using Theorems~\ref{approxWD_escaping} and~\ref{approxWD_oscillating}, we can obtain wandering domains contained in strips whose closures are not uniformly accessible,
hence the uniform accessibility assumption in Theorems~\ref{stripWDescaping} and~\ref{stripWDoscillating} is sufficient but not necessary. 
Similarly, many open sets containing a half-plane can arise as wandering domains without being lifts of sets via the exponential map. 
This raises the following question.

\begin{question}
    Denote the strip $S = [0,1] \times \R \subseteq \C$ and let $F \subseteq S$ be a closed set without holes such that $\dist(F, \partial S) > 0$.
    Does there always exist an entire function $f \in \mathcal{O}(\C)$ that has the connected components of $\mathring{F}$ as wandering domains?
    More generally, can the connected components of the interior of any closed set without holes be realized as wandering domain of some entire function?
\end{question}

Another question one might consider is the univalence of the iterates on unbounded wandering domains. 
Notice that Theorems~\ref{stripWDescaping} and~\ref{stripWDoscillating} produce an entire function whose iterates are univalent on the wandering domain.
On the other hand, a function we obtain via Theorem~\ref{liftWD} is $\infty$-to-$1$ on the unbounded wandering domain and injective on the other sets of the forward orbit of the wandering domain. 
Is it possible to realize an unbounded open set containing a half-plane as a wandering domain of some entire function such that the iterates of that function are univalent on this wandering domain?
In particular, this question can be asked for a half-plane itself.

\begin{question}
    Does there exist an entire function with a half-plane as a wandering domain whose iterates are univalent on that wandering domain?
\end{question}

\subsection*{Structure of the paper}

In Section~\ref{sec2}, we recall some basic notions of Arakelian approximation and prove several lemmas used throughout the paper.
In Section~\ref{sec3}, we introduce the notion of uniformly accessible closed sets and prove Theorems~\ref{stripWDescaping} and~\ref{stripWDoscillating}.
Section~\ref{sec4} is devoted to the proof of Theorem~\ref{liftWD}, while Section~\ref{sec5} contains the proof of Theorem~\ref{hornWD}.
Although the latter result is an immediate consequence of Theorems~\ref{approxWD_escaping} and~\ref{approxWD_oscillating},
proving Theorem~\ref{hornWD} directly will demonstrate the main idea behind the proofs of those theorems.
Finally, in Section~\ref{sec6}, we prove Theorems~\ref{approxWD_escaping},~\ref{approxWD_oscillating} and~\ref{combination}.

\subsection*{Notation}

Let $\Z$ denote the set of integers, $\N$ the set of positive integers, $\C$ the complex plane, and $\widehat{\C} = \C \cup \set{\infty}$ the one-point compactification of the complex plane. 

Given a point $z \in \C$ and $r > 0$, let $D(z,r)$ denote the closed disc centered at $z$ of radius $r$.
We also denote the distance between a point $z \in \C$ and a set $A \subseteq \C$ by 
\[\dist(z,A) = \inf_{w \in A} \abs{z-w},\]
and the distance between two sets $A,B \subseteq \C$ by
\[\dist(A,B) = \inf_{z \in A} \dist(z, B) = \inf \set{\abs{z-w} \mid z \in A,\, w \in B}.\]
For a set $A \subseteq \C$ and $r > 0$, we denote the \emph{$r$-neighbourhood of $A$} by
\[A[r] = \set{z \in \C \mid \dist(z,A) < r}.\] 

Given a set $A \subseteq \C$, we let $\mathring{A}$, $\partial A$, and $\overline{A}$ denote the topological interior, boundary, and closure of $A$, respectively.
Given two sets $A, B \subseteq \C$, we write $A \Subset B$ to mean $\overline{A} \subseteq \mathring{B}$. 
Furthermore, a family of sets $\set{A_j}_{j \in J}$ is said to be \emph{locally finite} if every point in $\C$ has an open neighbourhood that intersects only finitely many members of the family.

An unbounded simple curve is an continuous injective function $\gamma \colon [0,1) \to \C$ such that $\lim_{t \to 1}\gamma(t) = \infty$.
By a slight abuse of notation, we use $\gamma$ to denote both the function itself and its image in the complex plane, that is $\gamma = \gamma([0,1))$. 

For an open set $U \subseteq \C$, let $\mathcal{O}(U)$ denote the class of holomorphic functions on $U$.
More generally, for a set $A \subseteq \C$, let $\mathcal{O}(A)$ denote the class of holomorphic functions that are defined on some open neighbourhood of $A$.
Given a set $A \subseteq \C$ and $f \in \mathcal{O}(A)$, we define 
\[\norm{f}_A = \sup_{z \in A}\abs{f(z)}.\]
If $K \subseteq \C$ is a compact set, then $\norm{f}_K$ is just the usual supremum norm of $f \in \mathcal{O}(K)$.

Finally, it will be convenient to generalize the notion of simple connectedness to sets with multiple connected components.
We say that an open set is simply connected if each of its connected components is simply connected in the usual sense.
Under this convention, an open set $U$ is simply connected if and only if the set $\widehat{\C} \setminus U$ is connected, 
or equivalently, every connected component of $\C \setminus U$ is unbounded.

\subsection*{Acknowledgements}

I would like to sincerely thank Luka Boc Thaler for encouraging me to look into the geometry of unbounded wandering domains and for all the helpful discussions that followed. 
I would also like to thank Xavier Jarque for suggesting that I investigate wandering domains with complements of small area.

\section{Arakelian approximation} \label{sec2}

As mentioned in the introduction, the main tool in our constructions will be the Arakelian approximation theorem. 
In this section, we recall some basic notions of Arakelian approximation and prove several lemmas that will be used later. 
For a general overview of complex approximation theory, see the survey article~\cite{LegacyOfWeirstrass}.

Let $E \subseteq \C$ be a closed set. A bounded component of the complement $\C \setminus E$ is called a \emph{hole} of the set $E$.
As such, we say that the set $E$ has no holes if all the components of the complement $\C \setminus E$ are unbounded.
It is easy to see that $E$ has no holes precisely when the set $\widehat{\C} \setminus E$ is connected.
We write $h(E)$ to denote the union of the holes of the set $E$.
Notice that if $K \subseteq \C$ is a compact set, then the set $K$ has no holes, precisely when it is polynomially convex.
This is an immediate consequence of the Runge approximation theorem.

Let $E \subseteq \C$ be a closed set. We say that $E$ has the \emph{bounded exhaustion hull property}, or the \emph{BEH property}, 
if for every compact set $K \subseteq \C$ the set $h(E \cup K)$ is bounded. 
To verify the BEH property, it suffices to test the condition only for closed discs or for the elements of a compact exhaustion of $\C$.
Furthermore, with a little work, one can show that the set $E$ has the BEH property precisely when the set $\widehat{\C} \setminus E$ is locally connected at $\infty$.
Given a compact set $K$ and a closed set $E$ with the BEH property, it is easy to see that both sets $E \cup K$ and $E \cup h(E)$ have the BEH property. 
Notice that compact sets always enjoy the BEH property. For a non-example, consider the set, depicted in Figure~\ref{fig1}, given by
\[E = \set{0} \times [0,\infty) \cup \bigcup_{n \ge 1} \left(\set{\frac{1}{2^{n-1}}} \times [0,n] \cup \left[\frac{1}{2^n}, \frac{1}{2^{n-1}}\right] \times \set{n}\right).\]

\begin{figure}[H]
    \centering
    \includegraphics[width=0.44\textwidth]{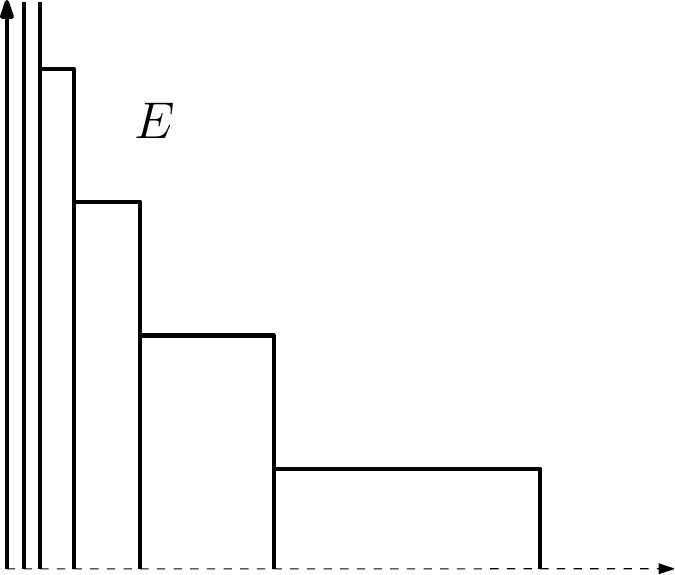}
    \caption{An example of a closed set $E$ that does not have the BEH property. 
    Indeed, given a closed disc $D$ centered at the origin, one can quickly verify that the set $h(E \cup D)$ is unbounded.}
    \label{fig1}  
\end{figure}

We say that a closed set $E \subseteq \C$ is \emph{Arakelian} if the set $E$ has no holes and enjoys the BEH property.
As noted above, this is equivalent to the set $\widehat{\C} \setminus E$ being connected and locally connected at $\infty$.
Arakelian showed that on such sets every holomorphic function can be approximated uniformly by entire functions. 
This result is known as the \emph{Arakelian approximation theorem}. 
For a proof, see the original paper by Arakelian~\cite{Arakelian} or~\cite[Theorem 10]{LegacyOfWeirstrass} for a modern reference.

\begin{theorem}\label{classical}
    Let $E \subseteq \C$ be an Arakelian set. Then for any holomorphic function $h \in \mathcal{O}(E)$ and $\varepsilon > 0$
    there exists an entire function $f \in \mathcal{O}(\C)$ such that 
    \[\norm{f-h}_{E} < \varepsilon.\]
\end{theorem}

This result is sharp in the following sense. 
A closed set $E \subseteq \C$ is called a set of \emph{uniform approximation} if the conclusion of the above theorem holds, 
namely that for any holomorphic function $h \in \mathcal{O}(E)$ and $\varepsilon > 0$ there exists an entire function $f \in \mathcal{O}(\C)$ such that $\norm{f-h}_{E} < \varepsilon$.
It turns out that a set of uniform approximation must be without holes and have the BEH property, see~\cite{Gauthier_Hengartner}.
Hence, Arakelian sets are precisely the sets of uniform approximation in the complex plane.

Recently, Fournodavlos gave another characterization of Arakelian sets.
In~\cite{ArakelianCharacterization}, he showed that a closed set is Arakelian if and only if it has a neighbourhood basis of simply connected open sets.
This characterization is particularly useful for proving statements involving Arakelian sets, although it is less convenient for checking whether a given set actually is Arakelian.  
Nevertheless, we will make frequent use of it throughout the paper.

\begin{theorem}\label{characterization}
    A closed set $E \subseteq \C$ is Arakelian if and only if for each open neighbourhood $V$ of $E$ there exists a simply connected open set $U$ such that $E \subseteq U \subseteq V$.
\end{theorem}

\begin{remark}
    Recall that a simply connected set might have multiple connected components, but every connected component is simply connected in the usual sense.
\end{remark}

An immediate consequence of the above theorem is that a finite union of pairwise disjoint Arakelian sets is again Arakelian.
Although this can be proven directly from the definitions, using the above characterization yields a very short and elegant proof, see~\cite[Corollary 2.9]{ArakelianCharacterization}.

\begin{proposition}
    A finite union of pairwise disjoint Arakelian sets is an Arakelian set.
\end{proposition}

We will also need a refined version of the Arakelian approximation theorem, which allows for interpolation at finitely many points.
This result is due to Nikolov and Pflug, see~\cite{ArakelianInterpolation}.

\begin{theorem}\label{interpolation}
    Let $E \subseteq \C$ be an Arakelian set, $h \in \mathcal{O}(E)$ and $\varepsilon > 0$.
    Furthermore, let $S \subseteq \mathring{E}$ be a finite set and $m \ge 0$. 
    Then there exists an entire function $f \in \mathcal{O}(\C)$ such that $\norm{f-h}_{E} < \varepsilon$ and $f^{(\nu)}(p) = h^{(\nu)}(p)$ for all $p \in S$ and $0 \le \nu \le m$.
\end{theorem}

Next, we will show that given a locally finite collection of pairwise disjoint Arakelian sets, 
functions can be approximated simultaneously on each set with independent precision. 
Moreover, interpolation at finitely many points in each of the sets is still possible.

\begin{proposition}\label{PosplositevArakelian}
    Let $\left(E_j\right)_{j \ge 1}$ be a locally finite pairwise disjoint family of Arakelian sets.
    For each $j \ge 1$, choose $h_j \in \mathcal{O}(E_k)$ and $\varepsilon_j > 0$.
    Furthermore, let $S_j \subseteq \mathring{E}_j$ be a finite set and let $m_j \ge 0$.
    Then there exists an entire function $f \in \mathcal{O}(\C)$ such that, for every $j \ge 1$, we have $\norm{f-h_j}_{E_j} < \varepsilon_j$ and $f^{(\nu)}(p) = h_j^{(\nu)}(p)$ for all $p \in S_j$ and $0 \le \nu \le m_j$.
\end{proposition}

\begin{proof}
    Let $\left(D_n\right)_{n \ge 1}$ be the sequence of closed discs $D_n = D(0,n)$.
    For each $n \ge 1$, we define an Arakelian set $\Pi_n$ as follows.
    Let $E_{j_1^n}, \ldots, E_{j^n_{l_n}}$ be those sets of the family $\left(E_j\right)_{j \ge 1}$ that intersect $D_n$.
    Since the family is locally finite, there are only finitely many.
    We then define 
    \[\Pi_n = D_n \cup \bigcup_{k=1}^{l_n}E_{j^n_k} \cup h\left(D_n \cup \bigcup_{k=1}^{l_n}E_{j^n_k}\right).\]
    By construction, the set $\Pi_n$ has no holes. Furthermore, since the sets $E_{j_1^n}, \ldots, E_{j^n_{l_n}}$ are Arakelian and pairwise disjoint, 
    their union has the BEH property. Since the set $D_n$ is compact, the union $D_n \cup \bigcup_{k=1}^{l_n}E_{j^n_k}$ also has the BEH property and then so does $\Pi_n$, thus $\Pi_n$ is an Arakelian set.
    Next we define 
    \[J_n = \set{j \in \N \mid E_j \cap \Pi_n \neq \emptyset}.\]
    Observe that all the sets $J_n$ are finite, $\N = \bigcup_{n \ge 1} J_n$ and $J_n \subseteq J_{n+1}$.
    Furthermore, $j \in J_n$ implies $E_j \subseteq \Pi_n$ and for all $j \in J_{n+1}\setminus J_n$ we have $E_j \cap \Pi_n = \emptyset$.
    Finally set $\widetilde{S}_0 = \emptyset$ and denote 
    \[\widetilde{S}_n = \bigcup_{j \in J_n}S_j \quad \text{and} \quad \widetilde{m}_n = \max \set{m_j \mid j \in J_n}.\]

    We will now inductively construct a sequence of entire functions $(f_n)_{n \ge 1}$ and a sequence of positive constants $(\delta_n)_{n \ge 1}$ 
    such that the following conditions hold:
    \begin{enumerate}
    \item $\norm{f_n-f_{n-1}}_{\Pi_{n-1}} < \delta_n$,
    \item $f_n^{(\nu)}(p) = f_{n-1}^{(\nu)}(p)$ for $p \in \widetilde{S}_{n-1}$ and $0 \le \nu \le \widetilde{m}_n$,
    \item $f_n^{(\nu)}(p) = h_j^{(\nu)}(p)$ for $p \in \widetilde{S}_n \setminus \widetilde{S}_{n-1}$ and $0 \le \nu \le \widetilde{m}_n$.
    \end{enumerate}
    
    For the case $n=1$, we first observe that the set $\widetilde{E}_1 = \bigcup_{j \in J_1} E_j$ is Arakelian. 
    We then denote $\delta_1 = \min\set{\varepsilon_j/2 \mid j \in J_1}$ and define the function $h \in \mathcal{O}\left(\widetilde{E}_1\right)$ by setting $h = h_j$ on $E_j$.
    Applying Theorem~\ref{interpolation} yields an entire function $f_1 \in \mathcal{O}(\C)$ such that 
    \[\norm{f_1 - h}_{\widetilde{E}_1} < \delta_1 \quad \text{and} \quad f_1^{(\nu)}(p) = h^{(\nu)}(p) \text{ for $p \in \widetilde{S}_1$ and $0 \le \nu \le \widetilde{m}_1$}.\]
    Condition (3) holds by construction, while conditions (1) and (2) in this case hold vacuously.
    
    For the general case, suppose we have already constructed $f_n \in \mathcal{O}(\C)$ and $\delta_n > 0$.
    Since we have $E_j \cap \Pi_n = \emptyset$ for $j \in J_{n+1} \setminus J_n$, the set $\widetilde{E}_{n+1} = \Pi_n \cup \bigcup_{j \in J_{n+1} \setminus J_n} E_j$ is a disjoint union of Arakelian sets and hence Arakelian.
    We then denote $\delta_{n+1} = \min \set{\delta_n/2,\, \varepsilon_j/2 \mid j \in J_{n+1} \setminus J_n}$ and define the function $h \in \mathcal{O}\left(\widetilde{E}_{n+1}\right)$ by setting $h = h_j$ on $E_j$ and $h = f_n$ on $\Pi_n$.
    Applying Theorem~\ref{interpolation} yields an entire function $f_{n+1} \in \mathcal{O}(\C)$ such that
    \[\norm{f_{n+1} - h}_{\widetilde{E}_{n+1}} < \delta_{n+1} \quad \text{and} \quad f_{n+1}^{(\nu)}(p) = \begin{cases}f_{n}^{(\nu)}(p) \text{ for $p \in \widetilde{S}_n$ and $0 \le \nu \le \widetilde{m}_{n+1}$} \\ h^{(\nu)}(p) \text{ for $p \in \widetilde{S}_{n+1} \setminus \widetilde{S}_n$ and $0 \le \nu \le \widetilde{m}_{n+1}$} \end{cases}.\]
    Note that conditions (1), (2), and (3) are satisfied by construction.

    Since $\delta_n < \delta_{n+1}/2$, condition (1) of the induction hypothesis implies that for any $n \ge 1$ the functions $(f_k)_{k \ge n}$ form a Cauchy sequence on $D_n$.
    As $\bigcup_{n \ge 1} D_n = \C$, the sequence $(f_n)_{n \ge 1}$ converges uniformly on compact sets to some entire function $f \in \mathcal{O}(\C)$.
    Now fix some $j \ge 1$ and let $N_j$ be the smallest integer such that $j \in J_{N_j}$.
    For $n \ge N_j$ we have $E_j \subseteq \Pi_{N_j} \subseteq \Pi_n$ and we can estimate
    \begin{align*}
        \norm{f-h_j}_{E_j} & \le \norm{f_{N_j}-h_j}_{E_j} + \sum_{n = N_j}^\infty \norm{f_{n+1}-f_n}_{E_j} \le \norm{f_{N_j}-h_j}_{E_j} + \sum_{n = N_j}^\infty \norm{f_{n+1}-f_n}_{\Pi_n} \\
        &< \sum_{n = N_j}^\infty \delta_{n} \le \sum_{n = N_j}^\infty \delta_j 2^{N_j - n} \le \varepsilon_j \sum_{n = N_j}^\infty 2^{N_j-n-1} = \varepsilon_j
    \end{align*}
    We also have $m_j \le \widetilde{m}_{N_j} \le \widetilde{m}_n$,
    so given a point $p \in S_j$ and $0 \le \nu \le m_j$, it follows from items (2) and (3) of the induction hypothesis that $f_n^{(\nu)}(p) = h_j^{(\nu)}(p)$ for all $n \ge N_j$.
    Since the functions $f_n$ converges to $f$ uniformly on compact sets, this implies $f^{(\nu)}(p) = h_j^{(\nu)}(p)$.
\end{proof}

\begin{remark}
    If the interpolation property is not required, the same proof can be carried out by using the classical Arakelian approximation theorem instead of Theorem~\ref{interpolation}.
    Furthermore, if the sets $(E_j)_{j \ge 1}$ are all compact, it suffices to use a Runge approximation theorem with interpolation, see \cite[Theorem 4]{GeometryOfWD}.
\end{remark}

In particular, Proposition~\ref{PosplositevArakelian} states that the union of a locally finite pairwise disjoint family of Arakelian sets is a set of uniform approximation.
In light of Arakelian's characterization, this gives the following proposition.

\begin{proposition}\label{locallyfiniteArakelian}
    The union of a locally finite pairwise disjoint family of Arakelian sets is an Arakelian set.
\end{proposition}

\begin{remark}
    The local finiteness of the family is a necessary condition. 
    Given a countable family of pairwise disjoint Arakelian, their union need not be an Arakelian set.
    Consider, for example, the family $\left(E_n\right)_{n \ge 0}$ given by
    \[E_0 = \set{2} \times [0,\infty), \quad E_n = \set{\sum_{k=0}^{n-1}\frac{1}{2^k}-\frac{1}{2^n}, \sum_{k=0}^{n-1}\frac{1}{2^k}} \times [0,n] \cup \left[\sum_{k=0}^{n-1}\frac{1}{2^k}-\frac{1}{2^n}, \sum_{k=0}^{n-1}\frac{1}{2^k}\right] \times \set{n}.\]
    It is easy to verify that each set $E_n$ is Arakelian, but their union $\bigcup_{n \ge 0} E_n$ is not, since it does not have the BEH property. See Figure~\ref{fig2}. 
\end{remark}

\begin{figure}[H]
    \centering
    \includegraphics[width=0.54\textwidth]{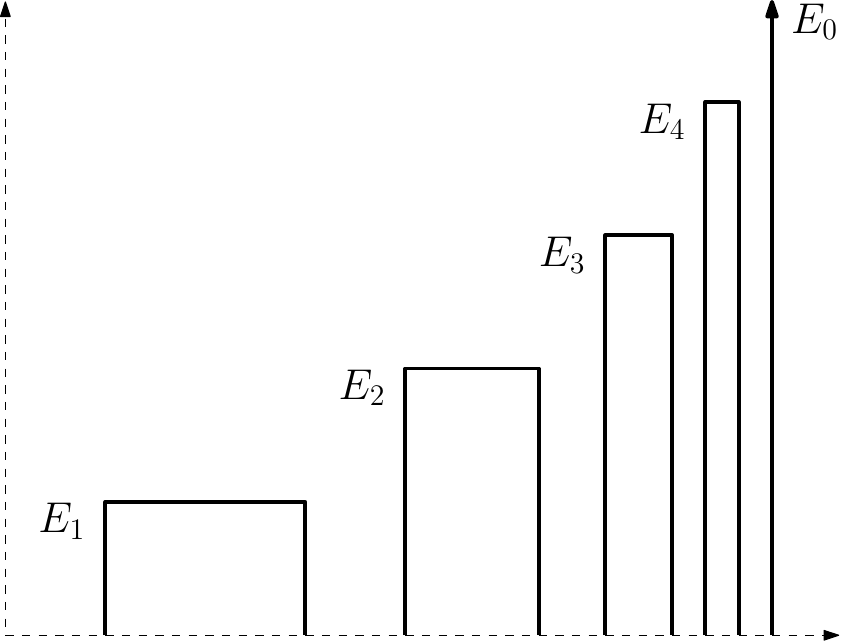}
    \caption{An example of a family of pairwise disjoint Arakelian sets $(E_n)_{n \ge 0}$ whose union is not an Arakelian set. 
    Indeed, given a closed disc $D$ centered at 2, one can quickly verify that the set $h\left(D \cup \bigcup_{n \ge 0} E_n\right)$ is unbounded.}
    \label{fig2}  
\end{figure}

The next proposition is a converse to the above statement. 
We show that an isolated connected component of an Arakelian set is an Arakelian set. 
Recall that a connected component $A_0$ of the set $A \subseteq \C$ is isolated if $A_0$ is an open set in the induced subset topology on $A$.

\begin{proposition}
    An isolated connected component of an Arakelian set is an Arakelian set.
\end{proposition}

\begin{proof}
    Let $E_0$ be an isolated connected component of the Arakelian set $E$, and choose $h \in \mathcal{O}(E_0)$ and $\varepsilon > 0$.
    Since $E_0$ is isolated, we can define a holomorphic function $\tilde{h} \in \mathcal{O}(E)$ by setting $\tilde{h} = h$ on $E_0$ and $\tilde{h} = 0$ on $E \setminus E_0$.
    As $E$ is Arakelian, it is a set of uniform approximation, and there exists an entire function $f \in \mathcal{O}(\C)$ such that $\norm{f - \tilde{h}}_E < \varepsilon$.
    In particular, $\norm{f - h}_{E_0} < \varepsilon$, so $E_0$ is also a set of uniform approximation and hence Arakelian.  
\end{proof}

\begin{remark}
    Together with Proposition~\ref{PosplositevArakelian}, the above result shows that in the Arakelian approximation theorem $\varepsilon$ can be taken to be a locally constant function.
    More precisely, given an Arakelian set $E$, a holomorphic function $h \in \mathcal{O}(E)$, and a locally constant function $\varepsilon \colon E \to (0, \infty)$, 
    there exists an entire function $f \in \mathcal{O}(\C)$ such that $\abs{f(z) - h(z)} < \varepsilon(z)$ for all $z \in E$.
\end{remark}

We finish the section by showing that, given a closed set contained in a simply connected open neighbourhood, 
one can enlarge it within the open neighbourhood so as to make it Arakelian. 
Furthermore, if the closed set has positive distance away from the boundary of the simply connected open neighbourhood,
it will also have positive distance away from the boundary of the Arakelian set.
The proof is based on ideas and constructions found in~\cite{UniformApproximationBoundedOnSet}.

\begin{proposition}\label{Arakelian_so_gosti}
    Let $F \subseteq U \subseteq \C$, where $F$ is a closed and $U$ a simply connected open set.
    Then there exists an Arakelian set $E$, with locally finite connected components, such that $F \Subset E \subseteq U$.
\end{proposition}

\begin{proof}
    For each $z \in F$, let $D_z$ denote a closed disc centered at $z$, such that $D_z \subseteq U$. 
    The family $\{\mathring{D}_z\}_{z \in F}$ forms an open cover of $F$, so we can choose a locally finite subcover $(\mathring{D}_{z_n})_{n \ge 1}$.
    We then define $H = \bigcup_{n \ge 1}D_{z_n}$ and set $E = H \cup h(H)$. 
    Note that $H$ is closed, since it is a locally finite union of closed sets.
    Furthermore, the connected components of $H$ are locally finite, so the same property holds for the set $E$.

    It is easy to check that $F \Subset E \subseteq U$, so it remains to show that the set $E$ is Arakelian.
    Note that $E$ has no holes by construction, so we must only very that it has the BEH property.
    To this end, let $B$ be a closed disc. 
    Since the family $(D_{z_n})_{n \ge 1}$ is locally finite, the disc $B$ intersects only finitely many of its members, say $D_{z_{n_1}}, \ldots, D_{z_{n_l}}$.
    The set $\partial B \setminus \bigcup_{k=1}^l D_{z_{n_k}}$ thus consist of finitely many disjoint open arcs.

    Let $U$ be a hole of the set $E \cup B$ and suppose that $\partial U$ does not intersect any of these arcs. 
    Then $\partial U \subseteq E$ and since $U \subseteq \C \setminus E$, this implies that $U$ is a hole of the set $E$, which is a contradiction.
    Thus $\partial U \cap I \neq \emptyset$, where $I$ is one of the open arcs making up the set $\partial B \setminus \bigcup_{k=1}^l D_{z_{n_k}}$. 
    We claim that $I \subseteq \partial U$.
    Indeed, since $\partial U$ is a closed subset of $\C$, the set $\partial U \cap I$ is closed in the subset topology of $I$.
    Next, pick $p \in \partial U \cap I$. In particular, $p \in \partial B \setminus E$, hence there exists an open neighbourhood $V$ of $p$ such that $V \cap E = \emptyset$ and that the set $V \setminus B$ is connected. 
    Since $(V \setminus B) \cap U \neq \emptyset$, the set $(V \setminus B) \cup U$ is a connected subset of $\C \setminus (E \cup B)$.
    Furthermore, since $U$ is a connected component of $\C \setminus (E \cup B)$, we have $V \setminus B \subseteq U$.
    This implies $V \cap I \subseteq \partial(V \setminus B) \cap I \subseteq \partial U \cap I$, which shows that $\partial U \cap I$ is open in the subset topology of $I$.
    Since the arc $I$ is connected, it follows that $\partial U \cap I = I$ or equivalently $I \subseteq \partial U$.

    Now suppose that $U_1$ and $U_2$ are two holes of the set $E \cup B$ and that $I \subseteq \partial U_1 \cap \partial U_2$.
    Pick $p \in I$ and let $V$ be an open neighbourhood of $p$, such that $V \cap E = \emptyset$ and that the set $V \setminus B$ is connected.
    Notice that $V \setminus B \subseteq \C \setminus (E \cup B)$.
    Then $W_1 = U_1 \cap (V \setminus B)$ and $W_2 = (\C \setminus (E \cup B \cup U_1)) \cap (V \setminus B)$ are disjoin open subsets of $V \setminus B$ such that $V \setminus B = W_1 \cup W_2$.
    Since $p \in V$ and $p \in \partial U_1 \cap \partial U_2$, we have $V \cap U_1 \neq \emptyset$ and $V \cap U_2 \neq \emptyset$.
    Thus $W_1 \neq \emptyset$ and $W_2 \neq \emptyset$, but this contradicts the connectedness of $V \setminus B$.
    
    To summarize, we have shown that any open arc making up the set $\partial B \setminus \bigcup_{k=1}^l D_{z_{n_k}}$ can be contained in the boundary of at most one of the holes of the set $E \cup B$. 
    Since there are only finitely many arcs, there can be at most finitely many holes, so the set $h(E \cup B)$ has to be bounded.
\end{proof}

\begin{corollary}
    Let $F \subseteq U \subseteq \C$, where $F$ is a closed and $U$ a simply connected open set such that $\dist(F, \partial U) > 0$.
    Then there exists an Arakelian set $E$, with locally finite connected components, such that $F \subseteq E \subseteq U$ and $\dist(F, \partial E) > 0$.
\end{corollary}

\begin{proof}
    Denote $r = \dist(F, \partial U) / 3$ and let $(z_n)_{n \ge 1}$ be a sequence of points in $F$ such that the family $(D(z_n, r))_{n \ge 1}$ forms a locally finite cover of $F$.
    We then set $D_{z_n} = D(z_n, 2r)$ and proceed as in the proof of Proposition~\ref{Arakelian_so_gosti}.
    The above construction now yields an Arakelian set $E$, with locally finite connected components, such that $F \subseteq E \subseteq U$ and with the additional property $\dist(F, \partial E) \ge r > 0$.
\end{proof}

\section{Prescribing the shape of wandering domains in a strip} \label{sec3}

The goal of this section is to prove Theorems~\ref{stripWDescaping} and~\ref{stripWDoscillating}.
Before doing so, we first introduce the notion of uniformly accessible closed sets. 
This geometric condition ensures that one can always find an open neighbourhood containing the closed set such that the closed set remains at a positive distance away from the boundary of the neighbourhood.
This will allow us to apply results like Proposition~\ref{unbounded_injectivity} or Koebe's 1/4-theorem, see~\cite[Corollary 1.4]{Pommerenke}.

We begin with the definition of uniformly accessible closed sets.

\begin{definition}\label{uniformly_accessible}
    Let $\rho > 0$. A closed set $F \subseteq \C$ is said to be \emph{$\rho$-uniformly accessible} 
    if every connected component $V$ of $\C \setminus F$ contains an unbounded simple curve $\gamma$ such that $\gamma[\rho] \subseteq V$.
    A closed set is \emph{uniformly accessible} if it is $\rho$-uniformly accessible for some $\rho > 0$.
\end{definition}

\begin{remark}
    It follows immediately from the definition that a uniformly accessible closed set has no holes,
    hence the connected components of its interior are simply connected.
    In particular, the wandering domains obtained in Theorems~\ref{stripWDescaping} and~\ref{stripWDoscillating} are simply connected.
\end{remark}

\begin{example}
    Consider the closed sets $E_1$ and $E_2$ given by 
    \begin{align*}
        E_1 &= (-\infty, 1] \times [-2,2] \cup \set{(x,y) \in (1,\infty) \times [-2,2] \Bigm\vert  1 \le \abs{y} \le 2} \\
        E_2 &= (-\infty, 1] \times [-2,2] \cup \set{(x,y) \in (1,\infty) \times [-2,2] \Bigm\vert  1/x \le \abs{y} \le 2}
    \end{align*}
    and depicted in Figure~\ref{fig_unif_access}.
    The set $E_1$ is clearly uniformly accessible; take, for example, $\rho = 1$. 
    The set $E_2$, on the other hand, is not, since the width of the cut-out region becomes arbitrarily small as we move to the right.
\end{example}

\begin{figure}[H]
    \centering
    \includegraphics[width=0.55\textwidth]{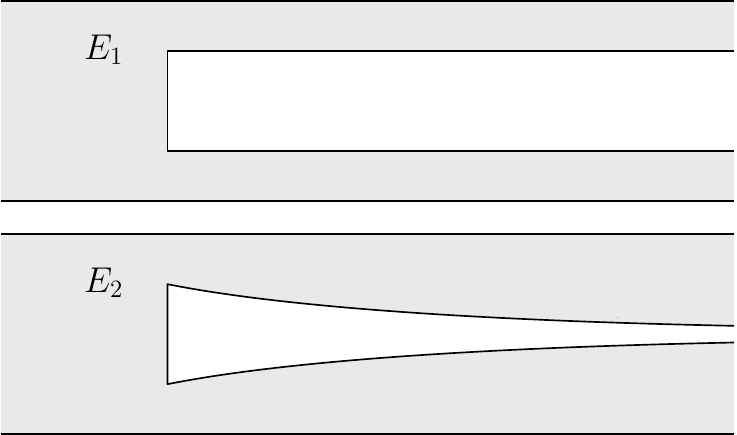}
    \caption{The set $E_1$ is uniformly accessible while the set $E_2$ is not.}
    \label{fig_unif_access} 
\end{figure}

\begin{remark}
    The notions of a closed set being Arakelian and being uniformly accessible are in general independent of one another.
    Indeed, the set $E$ defined at the beginning of Section~\ref{sec2} is uniformly accessible, but not Arakelian, since it does not have the BEH property.
    On the other hand, the set $E_2$ defined above is clearly Arakelian, but not uniformly accessible.
    Notice, however, that neither Arakelian sets nor uniformly accessible sets can have holes.  
\end{remark}

Next, we prove two lemmas regarding uniformly accessible sets that will be needed later.
The first lemma tells us how to extend a uniformly accessible set to an Arakelian set while avoiding a given point.
The second lemma shows that a uniformly accessible set can be approximated from above by a decreasing sequence of uniformly accessible sets.

\begin{lemma}\label{Strip_Arakelian}
    Let $U$ be a simply connected open set and let $F \subseteq U$ be a uniformly accessible closed set such that $\dist(F, \partial U) > 0$.
    Then, given a point $p \in \C \setminus F$, there exists an Arakelian set $E$ such that $F \subseteq E \subseteq U$, $\dist(F, \partial E) > 0$, and $p \in \C \setminus E$.
\end{lemma}

\begin{proof}
    Let $V$ be the connected component of $\C \setminus F$ that contains the point $p$.
    Since the set $F$ is uniformly accessible, there exists some $\rho > 0$ and some unbounded simple curve $\gamma$ in $V$, such that $\gamma[\rho] \subseteq V$.
    One can extend this curve to start in $p$ and so obtain a new unbounded simple curve $\eta$ in $V$ with $\dist(F, \eta) > 0$.
    By choosing $r < \rho$ small enough, we may assume that $\dist(F, \eta[r]) > 0$.
    We now consider the set $W = U \setminus \overline{\eta[r]}$. 
    Since $U$ is simply connected and $\overline{\eta[r]}$ is unbounded, the set $\C \setminus W = (\C \setminus U) \cup \overline{\eta[r]}$ only has unbounded connected components,
    so the set $W$ is simply connected.
    Furthermore, it follows from $\dist(F, \partial U) > 0$ and $\dist(F, \eta[r]) > 0$ that $\dist(F, \partial W) > 0$.
    By applying the corollary of Proposition~\ref{Arakelian_so_gosti} to the sets $F$ and $W$, we obtain an Arakelian set $E$ with all the desired properties. 
\end{proof}

\begin{lemma}\label{Okolice_Fn}
    Denote the strip $S = [0,1] \times \R \subseteq \C$ and let $F \subseteq S$ be a $\rho$-uniformly accessible closed set such that $\dist(F, \partial S) > 0$.
    Then there exists a sequence of closed sets $\left(F_n\right)_{n \ge 0}$ such that the following holds:
    \begin{enumerate}
        \item The sets $\left(F_n\right)_{n \ge 0}$ are $(\rho/2)$-uniformly accessible,
        \item $F_0 \subseteq S$ and $\dist(F_0, \partial S) > 0$,
        \item $F_{n+1} \Subset F_n$ for each $n \ge 0$,
        \item $\bigcap_{n \ge 0} F_n = F$.
    \end{enumerate}
\end{lemma}

\begin{proof}
    Let $(r_n)_{n \ge 0}$ be a strictly decreasing sequence of positive numbers converging to $0$, 
    and additionally assume, that $r_0 < \frac{1}{2}\min \set{\rho, \dist(F, \partial S)}$.
    For $n \ge 0$ we then define the closed set
    \[F_n = \overline{F[r_n]} \cup \Lambda_n,\]
    where $\Lambda_n$ denotes the union of all the connected components $V$ of the set $\C \setminus \overline{F[r_n]}$ for which there 
    exists no unbounded simple curve $\gamma$ in $V$ such that $\gamma[\rho/2] \subseteq V$.    
    It remains to verify that the sequence $\left(F_n\right)_{n \ge 0}$ has all the desired properties.

    \begin{enumerate}
        \item Fix some $n \ge 0$. It follows immediately from the construction that
        any connected component $V$ of the set $\C \setminus F_n$ is a connected component of the set $\C \setminus \overline{F[r_n]}$
        for which there exists an unbounded simple curve $\gamma$ in $V$ such that $\gamma[\rho/2] \subseteq V$.
        Thus, the set $F_n$ is $(\rho/2)$-uniformly accessible. 

        \item Since $F \subseteq S$ and 
        \[\dist \left(\overline{F[r_0]}, \partial S\right) \ge \dist(F, \partial S) - r_0 \ge \frac{1}{2} \dist(F, \partial S) > 0,\]
        we have $\overline{F[r_0]} \subseteq S$ or equivalently $\C \setminus S \subseteq \C \setminus \overline{F[r_0]}$.
        Let $V$ be a connected component of $\C \setminus \overline{F[r_0]}$ that contains one of the half-planes making up the set $\C \setminus S$.
        Then clearly $V$ contains an unbounded simple curve $\gamma$ such that $\gamma[\rho/2] \subseteq V$ and hence $V \subseteq \C \setminus F_0$.
        In particular, $\C \setminus S \subseteq \C \setminus F_0$ or equivalently $F_0 \subseteq S$.
        The later observation also implies $\dist(F_0, \partial S) = \dist \left(\overline{F[r_0]}, \partial S\right) > 0$.

        \item Fix some $n \ge 0$. Notice that $\overline{F[r_{n+1}]} \subseteq F[r_n] \subseteq \mathring{F}_n$, so we must only check $\Lambda_{n+1} \subseteq \mathring{F}_n$.
        It suffices to show $\Lambda_{n+1} \subseteq F_n$, so let $V$ be a connected component of $\C \setminus \overline{F[r_{n+1}]}$ making up the set $\Lambda_{n+1}$.
        Suppose that the intersection $V \cap (\C \setminus F_n)$ is non-empty and let $W$ be one of its connected components.
        Then $W \subseteq \C \setminus F_n$ is a connected set such that 
        \[\partial W \subseteq \partial V \cup \partial (\C \setminus F_n) \subseteq \overline{F[r_{n+1}]} \cup F_n \subseteq F_n.\]
        This implies that $W$ is a connected component of $\C \setminus F_n$ and thus contains an unbounded simple curve $\gamma$ such that $\gamma[\rho/2] \subseteq W$.
        It follows that $\gamma[\rho/2] \subseteq V$ which contradicts the assumption that $V \subseteq \Lambda_{n+1}$.
        Thus $V \cap (\C \setminus F_n) = \emptyset$ or equivalently $V \subseteq F_n$.

        \item By construction $F \subseteq F_n$ for all $n \ge 0$, hence $F \subseteq \bigcap_{n \ge 0} F_n$.
        To see the other inclusion, choose a point $p \in \C \setminus F$ and let $V$ be the connected component of $\C \setminus F$ that contains it.
        Since the set $F$ is $\rho$-uniformly accessible, there exists an unbounded simple curve $\gamma$ in $V$ such that $\gamma[\rho] \subseteq V$.
        One can extend this curve to start in $p$ and so obtain a new unbounded simple curve $\eta$ in $V$ with $\dist(\eta, F) > 0$.
        Let us denote $r = \dist(\eta, F)/2$.
        Since $r_n \searrow 0$, we can then choose $N \ge 0$ large enough such that $r_N < \min\set{r, \rho/2}$.
        It now follows that $\eta[r] \subseteq \C \setminus \overline{F[r_{N}]}$,
        and since the set $\eta[r]$ is connected, it is contained in a connected component $W$ of $\C \setminus \overline{F[r_{N}]}$.
        On the other hand, since $\rho < \dist(\gamma, F)$, we have $\gamma[\rho/2] \subseteq \C \setminus \overline{F[r_{N}]}$ and in particular $\gamma[\rho/2] \subseteq W$.
        Thus $W$ is not a component making up $\Lambda_N$ and so $\eta[r] \subseteq W \subseteq \C \setminus F_N$.
        In particular $p \in \C \setminus F_N$ which proves $F = \bigcap_{n \ge 0}F_n$. \qedhere
    \end{enumerate}
\end{proof}

The next lemma is a consequence of Proposition~\ref{unbounded_injectivity}. 
In particular, it gives sufficient conditions for an approximating function to map uniformly accessible sets to uniformly accessible ones.

\begin{lemma}\label{unbounded_injective_uniff_access}
    Let $U \subseteq \C$ be a simply connected open set, $L(z) = az + b$ a linear map, and $\lambda > 0$ some positive constant.
    Then there exist a constant $\delta > 0$, depending only on $L$ and $\lambda$, with the following property.
    Let $F \subseteq U$ be a uniformly accessible closed set with $\dist(F, \partial U) > 2\lambda$,
    and let $f \colon U \to \C$ be a holomorphic function satisfying $\norm{f-L}_U < \delta$. Then the following holds:
    \begin{enumerate}
        \item The function $f$ is injective on $A = \set{z \in U \mid \dist(z, \partial U) \ge \lambda}$,
        \item $f(F) \subseteq L(U)$ and $\dist(f(F), \partial L(U)) > 0$,
        \item the set $f(F)$ is closed and $\rho$-uniformly accessible, with $\rho$ independent of the function $f$.
    \end{enumerate}
\end{lemma}

\begin{proof}
    For a closed set $F$ as described above, it is easy to verify $F \subseteq A \subseteq U$ and $\dist(F, \partial A) \ge \lambda$.
    Furthermore, notice that $\C \setminus A = (\C \setminus U)[\lambda]$. 
    Since the set $U$ is simply connected, all the connected components of the set $\C \setminus U$ are unbounded, hence the same is true for the connected components of $(\C \setminus U)[\lambda]$.
    It follows that the set $A$ has no holes.
    
    Let $\delta_1 > 0$ be the constants obtained by applying Proposition~\ref{unbounded_injectivity} to the closed set $A \subseteq U$ and the map $L \colon U \to L(U)$.
    Furthermore, by using Cauchy estimates, we can choose $\delta_2 > 0$, such that $\norm{f - L}_U < \delta_2$ implies $\inf_{z \in A}\abs{f'(z)} \ge \abs{a}/2$.
    Finally, assume that the set $F$ is $\rho_0$-uniformly accessible. We then define 
    \[\delta = \min \set{\delta_1,\, \delta_2, \, \lambda|a|/3} \quad \text{and} \quad \rho = \min \set{\rho_0|a|/8 ,\,\lambda|a|/3}.\]
    We claim that $\delta$ is our desired constants. 
    Notice that $\delta$ only depend on the map $L$ and the constant $\lambda$, while $\rho$ only depends on the map $L$ and the constants $\lambda$ and $\rho_0$.

    Let $f \colon U \to \C$ be a holomorphic function such that $\norm{f-L}_U < \delta$. 
    By Proposition~\ref{unbounded_injectivity} the function $f$ is injective on $A$, $f(A) \subseteq L(U)$, and $\dist(f(A), \partial L(U)) > 0$.
    From Cauchy estimates it follows that $\inf_{z \in A}\abs{f'(z)} \ge \abs{a}/2$.
    Furthermore, notice that $\norm{f-L}_U < \delta$ implies that $f$ maps unbounded subsets of $U$ to unbounded sets.
    From this and the injectivity of $f$ on $A$, one can deduce that the set $f(A)$ is closed and that the restriction $f \vert_A \colon A \to f(A)$ is a homeomorphism.
    In particular, this implies $\partial f(A) = f(\partial A)$.

    It is clear that items (1) and (2) are satisfied, so it remains to verify item (3). We will show that the closed set $f(F)$ is $\rho$-uniformly accessible.
    To this end, let $V$ be a connected component of the set $\C \setminus f(F)$.
    Since $\partial V \subseteq f(F) \subseteq f(A)$, we either have $V \subseteq f(A)$ or that $V$ contains a connected component of the set $\C \setminus f(A)$.
    Treating these two cases will conclude the proof.

    First, suppose that $V \subseteq f(A)$. Since $f \vert_A$ is a homeomorphism and $V$ is a connected component of $\C \setminus f(F)$,
    it is easy to verify that the set $\left(f \vert_A\right)^{-1}(V)$ is a connected component of $\C \setminus F$.
    Since the set $F$ is $\rho_0$-uniformly accessible, the set $\left(f \vert_A\right)^{-1}(V)$ contains an unbounded simple curve $\gamma$ such that $\gamma[\rho_0] \subseteq \left(f \vert_A\right)^{-1}(V)$.
    It follows that $f \circ \gamma$ is an unbounded simple curve in $V$.
    Now let $p \in \gamma$. By assumption $f(\mathring{D}(p, \rho_0)) \subseteq V$, 
    so applying Koebe's 1/4-theorem together with the estimate $\abs{f'(p)} \ge |a|/2$ yields $\mathring{D}(f(p), \rho_0|a|/8) \subseteq f(\mathring{D}(p, \rho_0)) \subseteq V$.
    This shows that $(f \circ \gamma)[\rho] \subseteq V$.
    
    Next, suppose that $W$ is a connected component of the set $\C \setminus f(A)$ such that $W \subseteq V$.
    Since the set $A$ has no holes and $f \vert_A$ is a homeomorphism, the same has to hold for $f(A)$.
    It follows that the set $W$ is unbounded, hence it contains an unbounded simple curve $\gamma$.
    We claim that $\gamma[\dist(F, \partial A)|a|/3] \subseteq V$. 
    If this is not the case, there exists a point $p \in \partial V$ such that $\dist(p, \gamma) < \dist(F, \partial A)|a|/3$.
    Furthermore, since $\partial V \subseteq f(F)$ and $\gamma \subseteq \C \setminus f(A)$, 
    this implies $\dist(f(F), \partial f(A)) < \dist(F, \partial A)|a|/3$.
    On the other hand, we have the estimate
    \begin{align*}
        \dist(f(F), \partial f(A)) &= \dist(f(F), f(\partial A)) > \dist(L(F), L(\partial A)) - 2 \delta \\
        &= \dist(F, \partial A)|a| - 2 \delta \ge \dist(F, \partial A)|a|/3,
    \end{align*}
    which is a contradiction. Since $\rho \le \lambda|a|/3 \le \dist(F, \partial A)|a|/3$, this proves $\gamma[\rho] \subseteq V$.
\end{proof}

Finally, we will need the following lemma on the approximation of iterates. For a proof of the statement, see~\cite[Lemma 2.5]{EremenkoConj}.

\begin{lemma}\label{approxIter}
    Let $U \subseteq \C$ be an open set and let $h \in \mathcal{O}(U)$ be a holomorphic function. 
    Suppose that $F \subseteq U$ is a closed set and that $n \ge 1$ is such that $h^k(F)$ is defined and a subset of $U$ for all $0 \le k < n$.
    Denoting $\widehat{F} = \bigcup_{k=0}^{n-1} h^k(F)$, 
    suppose further that there exists an open neighbourhood $V \subseteq U$ of $\widehat{F}$,
    such that $\dist(\widehat{F}, \partial V) > 0$, and that $\abs{h'}$ is uniformly bounded above on $V$.
    Then for every $\varepsilon > 0$, there exists $\delta > 0$ with the following property.
    If $f \in \mathcal{O}(U)$ is a holomorphic function satisfying $\norm{f - h}_U < \delta$, then we have, for all $0 \le k \le n$, that 
    \[\norm{f^k - h^k}_F < \varepsilon.\] 
\end{lemma}

We are now ready to prove Theorems~\ref{stripWDescaping} and~\ref{stripWDoscillating}.

\begin{proof}[Proof of Theorem~\ref{stripWDescaping}]
    We denote $D = D(-3,1)$, $P_0 = \emptyset$, as well as 
    \[S_n = [-1 + 3(n-1), 1 + 3(n-1)] \times \R \quad \text{and} \quad P_n = [-3 - (1+3n), -3 + (1+3n)] \times \R\]
    for $n \ge 1$. Observe, that $S \subseteq S_1$, $\bigcup_{k = 1}^n S_k \subseteq P_n$, and $P_n \cap S_{n+1} = \emptyset$. 
    Next, we apply Lemma~\ref{Okolice_Fn} to the sets $F$ and $S$ to obtain a decreasing sequence of uniformly accessible closed sets $(F_n)_{n \ge 0}$ such that $\bigcap_{n \ge 0} F_n = F$.
    We also chose a sequence of points $(p_n)_{n \ge 1}$ such that $p_n \in \mathring{F}_{n-1} \setminus F_n$ and $\omega((p_n)_{n \ge 1}) = \partial F$, that is the sequence accumulates at every point of $\partial F$.
    Finally, we define the translation $\tau(z) = z+3$.

    Our goal will be to construct an entire function $f \in \mathcal{O}(\C)$ with the following properties:
    \begin{enumerate}[\indent (a)]
        \item $f^n(F) \subseteq S_{n+1}$ for all $n \ge 1$,
        \item $f(D) \subseteq \mathring{D}$ and $f^{n}(p_n) \in D$ for all $n \ge 1$,
        \item $f^n$ is injective on $F$ for all $n \ge 1$,
        \item $f^n$ maps unbounded subsets of $F$ to unbounded sets for all $n \ge 1$.
    \end{enumerate}
    Item (a) implies that the points in $F$ escape uniformly to infinity under iteration,
    so in particular, the set $\mathring{F}$ belongs to the Fatou set.
    On the other hand, item (b) implies that the set $D$ is contained in some attracting basin $\mathcal{A}$ and that $f^k(p_n) \in \mathcal{A}$ for all $k \ge n \ge 1$.
    Since the points $(p_n)_{n \ge 1}$ accumulate on $\partial F$, the family of iterates $(f^n)_{n \ge 1}$ cannot be normal on any neighbourhood of a boundary point of $F$, and $\partial F$ belongs to the Julia set.
    The same argument applies to the boundary of any iterate $f^n(F)$, hence the connected components of $\mathring{F}$ are escaping wandering domains.
    Finally, item (c) guarantees the injectivity of iterates, while item (d) ensures that unboundedness is preserved under iteration.
    Thus, the theorem holds, provided we can construct such a function.

    The function $f$ will be given as the limit of a sequence of entire functions $(f_n)_{n \ge 1}$ which will be constructed inductively using Arakelian approximation.
    At the $n$-th step of the construction, we will define an Arakelian set $\Pi_n$, an open set $V_n \subseteq \Pi_n$, and a holomorphic function $h_n \in \mathcal{O}(\Pi_n)$ as well as positive constants $\varepsilon_n > 0$ and $\lambda_n > 0$.
    The function $f_n$ will then be obtained by approximating the function $h_n$ on the set $\Pi_n$ up to an error of at most $\varepsilon_n > 0$.
    In particular, for $n \ge 1$, the following properties will hold:
    \begin{enumerate}
        \item $\varepsilon_1 < 1/4$ and $\varepsilon_{n+1} < \varepsilon_n/2$,
        \item $P_{n-1} \subseteq \Pi_n \subseteq P_n$,
        \item $\abs{f_n'}$ is uniformly bounded above on $V_n$ by some positive constant,
        \item for every $0 \le k \le n-1$ we have $f_n^k(F_n) \subseteq V_n$ with $\dist(f_n^k(F_n), \partial V_n) \ge \lambda_n$,
        \item $f_n^n(F_n) \subseteq S_{n+1}$ with $\dist(f_n^n(F_n), \partial S_{n+1}) > 0$,
        \item the function $f_n$ is injective on $V_n$,
        \item for every $k \ge n$ the set $f_n^n(F_k)$ is closed and uniformly accessible.
    \end{enumerate}
    Furthermore, if $g \in \mathcal{O}(\C)$ is an entire function satisfying $\norm{g-f_n}_{\Pi_n} < \varepsilon_n$, 
    it also enjoys properties (3) through (7). In particular, the bounds on $\abs{g'}$ are independent of the function $g$.

    The remainder of the proof is divided into three steps: the base case, the inductive step, and verifying that the limit of the functions $(f_n)_{n \ge 1}$ satisfies properties (a) through (d).

    \subsection*{The base case:}
    Recall that the closed set $F_1$ is uniformly accessible, $\dist(F_1, \mathring{S}_1) > 0$, and that the point $p_1 \in \mathring{S}_1$ is disjoint from the set $F_1$.
    We apply Lemma~\ref{Strip_Arakelian} to the point $p_1$ and the sets $F_1$ and $\mathring{S}_1$ to obtain an Arakelian set $E_1$ such that $F_1 \subseteq E_1 \subseteq \mathring{S}_1$, $\dist(F_1, \partial E_1) > 0$, and $p_1 \in \C \setminus E_1$.
    We then denote $\lambda_1 = \dist(F_1, \partial E_1)/3$ and define the sets 
    \[A_1 = \set{z \in E_1 \mid \dist(z, \partial E_1) \ge \lambda_1} \quad \text{and} \quad V_1 = \mathring{A}_1.\]
    Note $\dist(A_1, \partial E_1) = \lambda_1$ and $\dist(F_1, \partial A_1) \ge 2\lambda_1$. 
    Since $p_1 \in \mathring{S}_1$, we can choose a closed disc $B_1 \subseteq \mathring{S}_1 \setminus E_1$ centered in $p_1$.
    Finally, we define the set 
    \[\Pi_1 = D \cup E_1 \cup B_1.\]
    Since $D$, $E_1$, and $B_1$ are pairwise disjoint Arakelian sets, their union $\Pi_1$ is also Arakelian.
     
    Next, we define the function $h_1 \colon \Pi_1 \to \C$ by setting
    \[h_1(z) = \begin{cases} -3;& z \in D \\ z+3;& z \in E_1 \\ -3;& z \in B_1 \end{cases},\]
    see Figure~\ref{StripEscWD}.
    Notice, that the functions $h_1$ extends holomorphically to an open neighbourhood of $\Pi_1$, that is $h_1 \in \mathcal{O}(\Pi_1)$.

    \begin{figure}[ht]
        \centering
        \includegraphics[width=0.8\textwidth]{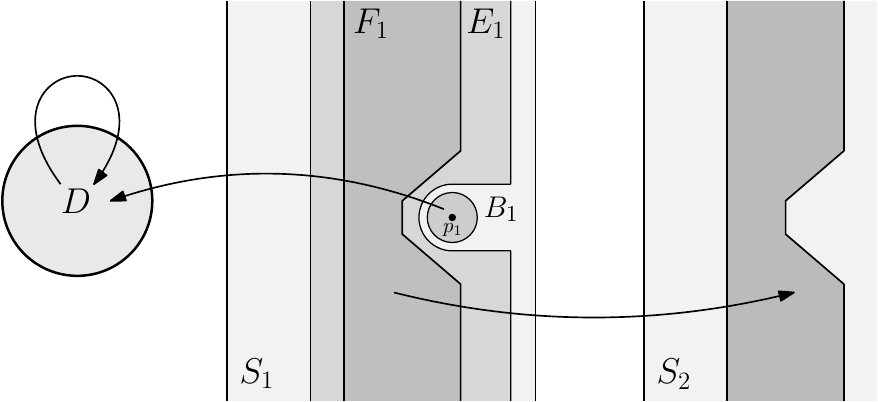}
        \caption{The construction of the map $h_1$ in the proof of Theorem~\ref{stripWDescaping}.}
        \label{StripEscWD} 
    \end{figure}
    
    We now chose $0 < \varepsilon_1 < 1/4$ small enough, so that if an entire function $g \in \mathcal{O}(\C)$ satisfies $\norm{g- h_1}_{\Pi_1} < 2\varepsilon_1$ the following holds:
    \begin{enumerate}[\indent (i)]
        \item $\abs{g'}$ is uniformly bounded above on $V_1$ by some positive constant,
        \item $g(D) \subseteq \mathring{D}$ and $g(p_1) \in D$,
        \item $g(F_1) \subseteq S_2$ with $\dist(g(F_1), \partial S_2) > 0$,
        \item the function $g$ is injective on $V_1$,
        \item for every $k \ge 1$ the set $g(F_k)$ is closed and uniformly accessible.
    \end{enumerate} 
    Since $\dist(V_1, \partial E_1) = \lambda_1$ and $\abs{h_1'} = 1$ on $E_1$, item (i) can be guaranteed by using Cauchy estimates.
    Item (ii) holds, since we chose $\varepsilon_1 < 1/4$. 
    Finally, we apply Lemma~\ref{unbounded_injective_uniff_access} to the open set $\mathring{E}_1$, the translation $\tau$, and the constant $\lambda_1$. 
    Notice that this is possible, since the set $E_1$ is Arakelian, so the set $\mathring{E}_1$ is simply connected. 
    By taking $\varepsilon_1$ small enough so that the conclusion of the lemma holds, items (iii), (iv), and (v) are then also satisfied.
    Note, that the bounds on $\abs{g'}$ do not depend on the function $g$.
    
    Finally, we use the Arakelian approximation theorem to obtain an entire function $f_1 \in \mathcal{O}(\C)$ such that $\norm{f_1 - h_1}_{\Pi_1} < \varepsilon_1$.
    With that, items (1) through (7) are satisfied for the case $n=1$.

    \subsection*{The inductive step:} Suppose, we have already completed the $n$-th step of the inductive proces.
    Item (5) of the inductive hypothesis tells us that $f^n_n(F_{n}) \subseteq S_{n+1}$ with $\dist(f^n_n(F_{n}), \partial S_{n+1}) > 0$,
    while items (4) and (6) imply that the function $f_n^n$ is injective on $F_n$, hence $f^n_n(F_{n+1})$ and $f^n_n(p_{n+1})$ are disjoint subsets of $f^n_n(F_n)$.
    Finally, by item (7) of the induction hypothesis, the set $f^n_n(F_{n+1})$ is closed and uniformly accessible.
    
    We apply Lemma~\ref{Strip_Arakelian} to the point $f^n_n(p_{n+1})$ and the sets $f^n_n(F_{n+1})$ and $\mathring{S}_{n+1}$ to obtain an Arakelian set $E_{n+1}$ such that $f^n_n(F_{n+1}) \subseteq E_{n+1} \subseteq \mathring{S}_{n+1}$, $\dist(f^n_n(F_{n+1}), \partial E_{n+1}) > 0$, and $f^n_n(p_{n+1}) \in \C \setminus E_{n+1}$.
    We then denote 
    \[\lambda_{n+1} = \min \set{\dist(f^n_n(F_{n+1}), \partial E_{n+1})/3, \lambda_n}\]
    and define the sets 
    \[A_{n+1} = \set{z \in E_{n+1} \mid \dist(z, \partial E_{n+1}) \ge \lambda_{n+1}} \quad \text{and} \quad V_{n+1} = V_n \cup \mathring{A}_{n+1}.\]
    Note $\dist(A_{n+1}, \partial E_{n+1}) = \lambda_{n+1}$ and $\dist(f^n_n(F_{n+1}), \partial A_{n+1}) \ge 2\lambda_{n+1}$.
    Since $f^n_n(p_{n+1}) \in \mathring{S}_{n+1}$, we can choose a closed disc $B_{n+1} \subseteq \mathring{S}_{n+1} \setminus E_{n+1}$ centered in $f^n_n(p_{n+1})$.
    Finally, we define the set 
    \[\Pi_{n+1} = P_n \cup E_{n+1} \cup B_{n+1}.\]
    Since $P_n$, $E_{n+1}$ and $B_{n+1}$ are pairwise disjoint Arakelian sets, their union $\Pi_{n+1}$ is also Arakelian.

    Next, we define the function $h_{n+1} \colon \Pi_{n+1} \to \C$ by setting
    \[h_{n+1}(z) = \begin{cases} f_n;& z \in P_n \\ z+3;& z \in E_{n+1} \\ -3;& z \in B_{n+1} \end{cases}.\]
    Notice, that the functions $h_{n+1}$ extends holomorphically to an open neighbourhood of $\Pi_{n+1}$, that is $h_{n+1} \in \mathcal{O}(\Pi_{n+1})$.
    
    We now chose $0 < \varepsilon_{n+1} < \varepsilon_n/2$ small enough, so that if an entire function $g \in \mathcal{O}(\C)$ satisfies $\norm{g- h_{n+1}}_{\Pi_{n+1}} < 2\varepsilon_{n+1}$, the following holds:
    \begin{enumerate}[\indent (i)]
        \item $\abs{g'}$ is uniformly bounded above on $V_{n+1}$ by some positive constant,
        \item $g^{n+1}(p_{n+1}) \in D$,
        \item for every $0 \le k \le n$ we have $g^k(F_{n+1}) \subseteq V_{n+1}$ with $\dist(g^k(F_{n+1}), \partial V_{n+1}) \ge \lambda_{n+1}$,
        \item $g^{n+1}(F_{n+1}) \subseteq S_{n+2}$ with $\dist(g^{n+1}(F_{n+1}), \partial S_{n+2}) > 0$,
        \item the function $g$ is injective on $V_{n+1}$,
        \item for every $k \ge n+1$ the set $g^{n+1}(F_k)$ is closed and uniformly accessible.
    \end{enumerate} 
    Since $\Pi_n \subseteq P_n$ and $\varepsilon_{n+1} < \varepsilon_n/2$, the function $g$ will satisfy $\norm{g- f_n}_{\Pi_{n}} < \varepsilon_{n}$,
    so by the induction hypothesis, it will enjoy properties (3) through (7).
   
    For item (i), notice that item (3) of the induction hypothesis implies that $\abs{g'}$ is uniformly bounded above on $V_{n}$,
    while the uniform bound on $\mathring{A}_{n+1}$ can be guaranteed by using Cauchy estimates, since $\dist(A_{n+1}, \partial E_{n+1}) = \lambda_{n+1}$ and $\abs{h_{n+1}'} = 1$ on $E_{n+1}$.
   
    Notice that $\abs{h_{n+1}'}$ is uniformly bounded above on $V_n$ by item (3) of the induction hypothesis and that we have 
    \[\widehat{F}_n = \bigcup_{k=0}^{n-1}h^k_{n+1}(F_n) \subseteq V_n \subseteq \mathring{\Pi}_{n+1}\] 
    with $\dist(\widehat{F}_n, V_n) \ge \lambda_n$ by item (4) of the induction hypothesis.
    Thus we may apply Lemma~\ref{approxIter} and by choosing $\varepsilon_{n+1}$ sufficiently small, we can guarantee that 
    \[g^n(p_{n+1}) \in B_{n+1} \quad \text{and} \quad \norm{g^n - f_n^n}_{F_{n+1}} < \lambda_{n+1}.\]
    Since $h_{n+1} = -3$ on $B_{n+1}$ and $\varepsilon_{n+1} < 1/4$ by item (1) of the induction hypothesis, the former implies $g^{n+1}(p_{n+1}) \in \mathring{D}$ and so item (ii) holds.
    On the other hand, since $g^k(F_{n+1}) \subseteq V_n$ with $\dist(g^k(F_{n+1}), \partial V_n) \ge \lambda_n \ge \lambda_{n+1}$ holds for all $0 \le k \le n-1$ by item (4) of the induction hypothesis,
    the later estimate together with ${\dist(f^n_n(F_{n+1}), \partial A_{n+1}) \ge 2\lambda_{n+1}}$ implies that $g^n(F_{n+1}) \subseteq \mathring{A}_{n+1}$ with $\dist(g^n(F_{n+1}), \partial A_{n+1}) \ge \lambda_{n+1}$,
    thus proving item (iii). 

    Next, we apply Lemma~\ref{unbounded_injective_uniff_access} to the open set $\mathring{E}_{n+1}$, the translation $\tau$, and the constant $\lambda_{n+1}$.
    Notice that this is possible, since the set $E_{n+1}$ is Arakelian and so the set $\mathring{E}_{n+1}$ is simply connected. 
    We take $\varepsilon_{n+1}$ small enough, so that the conclusion of the lemma holds.
    It then immediately follows that a holomorphic function $g$ satisfying $\norm{g- h_{n+1}}_{\Pi_{n+1}} < 2\varepsilon_{n+1}$ will be injective on $\mathring{A}_{n+1}$,
    and since it is injective on $V_n$ by item (6) of the induction hypothesis, we have item (v).
    It also follows that $g^{n+1}(F_{n+1}) \subseteq S_{n+2}$ with $\dist(g^{n+1}(F_{n+1}), \partial S_{n+2}) > 0$, so item (iv) holds.
    Since $g^n(F_k)$ is closed and uniformly accessible for all $k \ge n+1$ by item (7) of the induction hypothesis, item (vi) is also immediate.
    Note, that the bounds on $\abs{g'}$ do not depend on the function $g$.
    
    Finally, we use the Arakelian approximation theorem to obtain an entire function $f_{n+1} \in \mathcal{O}(\C)$ such that $\norm{f_{n+1} - h_{n+1}}_{\Pi_{n+1}} < \varepsilon_{n+1}$.
    From the above construction, it is clear that items (1) through (7) are satisfied for the case $n + 1$.
    
    \subsection*{Verification:} Condition (1) implies that for every $n \ge 1$ the functions $(f_k)_{k \ge n}$ form a Cauchy sequence on $\Pi_n$, and by condition (2) we have $\bigcup_{n=1}\Pi_n = \C$.
    Thus, the sequence $(f_n)_{n \ge 1}$ converges uniformly on compact sets to some entire function $f \in \mathcal{O}(\C)$.
    Given $n \ge 1$ we then have $f = f_n + \sum_{k=n}^\infty f_{k+1}-f_k$, which yields the estimate 
    \[\norm{f - f_n}_{\Pi_n} \le \sum_{k=n}^\infty \norm{f_{k+1}-f_k}_{\Pi_n} < \sum_{k=n}^\infty \varepsilon_{k+1} < \varepsilon_n \sum_{k=n}^\infty 2^{k+1-n} < \varepsilon_n,\]
    so by the induction hypothesis, the function $f$ enjoys the same properties as $f_n$.
    In particular, we have $f^n(F) \subseteq S_{n+1}$ for all $n \ge 1$, $f(D) \subseteq \mathring{D}$, and $f^n(p_n) \in D$ for all $n \ge 1$, thus proving items (a) and (b).
    Since $f^k(F) \subseteq V_n$ for all $0 \le k \le n-1$ and $f$ is injective on $V_n$, it follows that $f^n$ is injective on the set $F$, so item (c) holds.
    Finally, since 
    \[\norm{f - \tau}_{V_n} \le \norm{f - h_n}_{E_n} \le \norm{f - f_n}_{E_n} + \norm{f_n - h_n}_{E_n} < 2\varepsilon_n,\]
    the function $f$ is uniformly close to a translation on $V_n$ and therefore maps unbounded subsets of $V_n$ to unbounded sets.
    In particular, this implies that the function $f^n$ maps unbounded subsets of $F$ to unbounded ones, proving item (d). 
\end{proof}

\begin{proof}[Proof of Theorem~\ref{stripWDoscillating}]
    We denote $D = D(-3,1)$, $P_0 = \emptyset$, as well as 
    \[S_n = [-1 + 3(n-1), 1 + 3(n-1)] \times \R \quad \text{and} \quad P_n = [-3 - (1+3n), -3 + (1+3n)] \times \R\]
    for $n \ge 1$. Observe that $\bigcup_{k = 1}^n S_k \subseteq P_n$ and $P_n \cap S_{n+1} = \emptyset$. 
    Next, we define the recursive sequences 
    \begin{alignat*}{2}
        c_{n+1} &= c_n + d_n/4, \quad & c_1 &= -1, \\
        d_{n+1} &= d_n/4,  & \quad d_1 &= 1,
    \end{alignat*}
    and denote the intervals 
    \[I_n = [c_n, c_n + 2d_n] \quad I_n^- = [c_n, c_n + d_n] \quad I_n^+ = [c_n + d_n, c_n + 2d_n],\]
    notice $I_n = I_n^- \cup I_n^+$. We then go on to define 
    \[C_n = I_n \times \R \qquad  C_n^{+} = I_n^{+} \times \R \qquad C_n^{-} = I_n^{-} \times \R\]
    and observe that $C_n = C_n^- \cup C_n^+$, $C_1 = S_1$, and $C_1^+ = S$.
    Furthermore, since 
    \[I_{n+1} = [c_{n+1}, c_{n+1} + 2d_{n+1}] = [c_n + d_n/4, c_n + 3d_n/4] \subseteq [c_n, c_n + d_n] = I_n^-,\]
    it follows that $C_{n+1} \subseteq C_n^-$ with $\dist(C_{n+1}, \partial C_n^-) > 0$.
    We fix a point $q \in F$ and choose a sequence $(q_n)_{n \ge 1}$ such that $q_n \in I_{n+1}^+$, so in particular $\lim_{n \to \infty} q_n = -2/3$.    
    We apply Lemma~\ref{Okolice_Fn} to the sets $F$ and $S$ to obtain a decreasing sequence of uniformly accessible closed sets $(F_n)_{n \ge 0}$ such that $\bigcap_{n \ge 0} F_n = F$.
    We also chose a sequence of points $(p_n)_{n \ge 1}$ such that $p_n \in \mathring{F}_{n-1} \setminus F_n$ and $\omega((p_n)_{n \ge 1}) = \partial F$, that is the sequence accumulates at every point of $\partial F$.
    Finally, we define the translation $\tau(z) = z+3$ and set $N_n = n(n+1)/2 = 1 + 2 + \ldots + n$.

    Our goal will be to construct an entire function $f \in \mathcal{O}(\C)$ with the following properties:
    \begin{enumerate}[\indent (a)]
        \item $f^{N_m + k}(F) \subseteq S_{k+1}$ for all $0 \le k \le m$,
        \item we have $\lim_{n \to \infty}f^{N_n} = -2/3$ uniformly on compact subsets of $F$,
        \item $f(D) \subseteq \mathring{D}$ and $f^{N_n}(p_n) \in D$ for all $n \ge 1$,
        \item $f^n$ is injective on $F$ for all $n \ge 1$,
        \item $f^n$ maps unbounded subsets of $F$ to unbounded sets for all $n \ge 1$.
    \end{enumerate}
    Let $(f^{n_m})_{m \ge 1}$ be a sequence of iterates and for each $k \ge 1$ define the set 
    \[J_k = \set{m \ge 1 \mid f^{n_m}(F) \subseteq S_k}.\]
    Note that item (a) implies $\N = \bigcup_{k \ge 1} J_k$.
    If all the sets $J_k$ are finite, the sequence $(f^{n_m})_{m \ge 1}$ converges to infinity uniformly on compact subsets of $F$.
    Otherwise, there exists some $k \ge 1$, such that the set $J_k$ is infinite.
    It follows from item (a) that $f^{n_m - k}(F) \subseteq S_1$ so we have $n_m - k \in \set{N_n \mid n \ge 1}$ for all $m \in J_k$.
    Item (b) and the continuity of $f$ then imply that the subsequence $(f^{n_m})_{m \in J_k}$ converges to $f^k(-2/3)$ uniformly on compact subsets of $F$. 
    Considering both cases, we see that $\mathring{F}$ belongs to the Fatou set.
    On the other hand, item (c) implies that the set $D$ is contained in some attracting basin $\mathcal{A}$, and that $f^k(p_n) \in \mathcal{A}$ for all $k \ge N_n$.
    Since the points $(p_n)_{n \ge 1}$ accumulate on $\partial F$, the family of iterates $(f^n)_{n \ge 1}$ cannot be normal on any neighbourhood of a boundary point of $F$, and $\partial F$ belongs to the Julia set.
    The same argument applies to the boundary of any iterate $f^n(F)$, hence the connected components of $\mathring{F}$ are oscillating wandering domains.
    Finally, item (d) guarantees the injectivity of iterates, while item (e) ensures that unboundedness is preserved under iteration.
    Thus, the theorem holds, provided we can construct such a function.

    The function $f$ will be given as the limit of a sequence of entire functions $(f_n)_{n \ge 1}$ which will be constructed inductively using Arakelian approximation.
    At the $n$-th step of the construction, we will define an Arakelian set $\Pi_n$, an open set $V_n \subseteq \Pi_n$, and a holomorphic function $h_n \in \mathcal{O}(\Pi_n)$ as well as positive constants $\varepsilon_n > 0$ and $\lambda_n > 0$.
    The function $f_n$ will then be obtained by approximating the function $h_n$ on the set $\Pi_n$ up to an error of at most $\varepsilon_n > 0$.
    In particular, for $n \ge 1$, the following properties will hold:
    \begin{enumerate}
        \item $\varepsilon_1 < 1/4$ and $\varepsilon_{n+1} < \varepsilon_n/2$,
        \item $P_{n-1} \subseteq \Pi_n \subseteq P_n$,
        \item $\abs{f_n'}$ is uniformly bounded above and below on $V_n$ by some positive constants,
        \item $f_n^k(C_{n+1}) \subseteq S_{k+1}$ for all $0 \le k \le n$ and $f_n^{N_m + k}(F_n) \subseteq S_{k+1}$ for all $0 \le k \le m \le n$,
        \item $f_n^k(C_{n+1}) \subseteq V_n$ with $\dist(f_n^k(C_{n+1}), \partial V_n) \ge \lambda_n$ for all $0 \le k \le n-1$ and $f_n^{k}(F_n) \subseteq V_n$ with $\dist(f_n^{k}(F_n), \partial V_n) \ge \lambda_n$ for all $0 \le k \le N_n+n-1$,
        \item $f_n^{N_n}(F_n) \subseteq C_{n+1}^+$ with $\dist(f_n^{N_n}(F_n), \partial C_{n+1}^+) > 0$,
        \item the function $f_n$ is injective on $V_n$ and maps uniformly accessible subsets of $V_n$ to uniformly accessible ones,
        \item for every $k \ge n$ the set $f_n^{N_n +n}(F_k)$ is closed and uniformly accessible.
    \end{enumerate}
    Furthermore, if $g \in \mathcal{O}(\C)$ is an entire function satisfying $\norm{g-f_n}_{\Pi_n} < \varepsilon_n$, 
    it also enjoys properties (3) through (8). In particular, the bounds on $\abs{g'}$ are independent of the function $g$.

    The remainder of the proof is divided into three steps: the base case, the inductive step, and verifying that the limit of the functions $(f_n)_{n \ge 1}$ satisfies properties (a) through (e).

    \subsection*{The base case:}
    First recall that by definition $C_2 \subseteq \mathring{C}_1^-$ with $\dist(C_2, \partial \mathring{C}_1^-) > 0$ and that $F_1 \subseteq \mathring{C}_1^+$ with $\dist(F_1, \partial \mathring{C}_1^+) > 0$.
    Furthermore, the set $F_1$ is uniformly accessible and the point $p_1 \in \mathring{C}_1^+$ is disjoint from the set $F_1$.
    We apply the corollary of Proposition~\ref{Arakelian_so_gosti} to the sets $C_2$ and $\mathring{C}_1^-$ to obtain an Arakelian set $E_1^-$ such that $C_2 \subseteq E_1^- \subseteq \mathring{C}_1^-$ and $\dist(C_2, \partial E_1^-) > 0$.  
    We also apply Lemma~\ref{Strip_Arakelian} to the point $p_1$ and sets $F_1$ and $\mathring{C}_1^+$ to obtain an Arakelian set $E_1^+$ such that $F_1 \subseteq E_1^+ \subseteq \mathring{C}_1^-+$, $\dist(F_1, \partial E_1^+) > 0$, and $p_1 \in \C \setminus E_1^+$.
    Observe that the Arakelian sets $E_1^+$ and $E_1^-$ are disjoint. We then denote 
    \[\lambda_1 = \min \set{\dist(C_2, \partial E_1^-)/3,\, \dist(F_1, \partial E_1^+)/3}\]
    and define the sets 
    \[A_1 = \set{z \in E_1^- \cup E_1^+ \mid \dist(z, \partial E_1^- \cup \partial E_1^+) \ge \lambda_1} \quad \text{and} \quad V_1 = \mathring{A}_1.\]
    Note $\dist(A_1, \partial E_1^- \cup \partial E_1^+) = \lambda_1$ and $\dist(F_1 \cup C_2, \partial A_1) \ge 2\lambda_1$.
    Since $p_1 \in \mathring{C}_1^+$, we can choose a closed disc $B_1 \subseteq \mathring{C}_1^+ \setminus E_1^+$ centered in $p_1$.
    Finally, we define the set 
    \[\Pi_1 = D \cup E_1^- \cup E_1^+ \cup B_1.\]
    Since $D$, $E_1^-$, $E_1^+$, and $B_1$ are pairwise disjoint Arakelian sets, their union $\Pi_1$ is also Arakelian.

    We choose a linear map $L_1$, such that $L_1(E_1^+) \subseteq C_2^+$ with $\dist(L_1(E_1^+), \partial C_2^+) > 0$, $L_1(q) = q_1$, and that $\abs{L_1'} < 1/2$. 
    We then define the function $h_1 \colon \Pi_1 \to \C$ by setting
    \[h_1(z) = \begin{cases} -3;& z \in D \\ z+3;& z \in E_1^- \\ L_1(z);& z \in E_1^+ \\ -3;& z \in B_1 \end{cases},\]
    see Figure~\ref{StripOscWD}.
    Note, that the function $h_1$ extends holomorphically to an open neighbourhood of $\Pi_1$, that is $h_1 \in \mathcal{O}(\Pi_1)$.

    \begin{figure}[ht]
        \centering
        \includegraphics[width=0.85\textwidth]{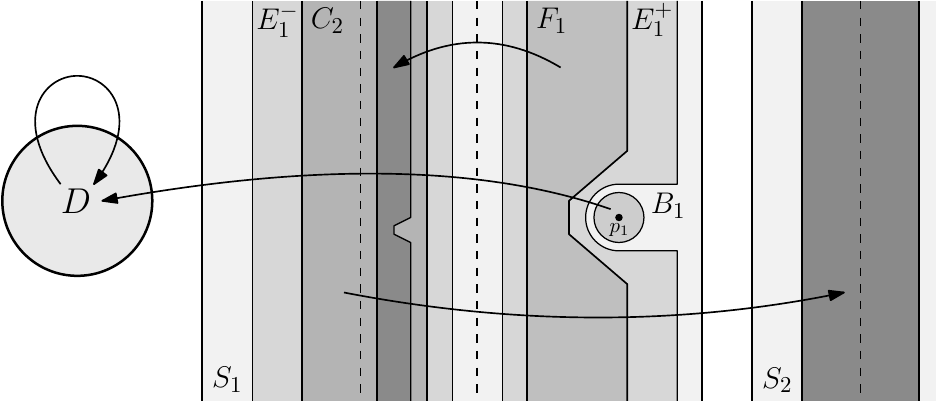}
        \caption{The construction of the map $h_1$ in the proof of Theorem~\ref{stripWDoscillating}.}
        \label{StripOscWD} 
    \end{figure}

    We now chose $0 < \varepsilon_1 < 1/4$ small enough, so that if an entire function $g \in \mathcal{O}(\C)$ satisfies $\norm{g- h_1}_{\Pi_1} < 2\varepsilon_1$ the following holds:
    \begin{enumerate}[\indent (i)]
        \item $\abs{g'}$ is uniformly bounded above and below on $V_1$ by some positive constants,
        \item $g(D) \subseteq \mathring{D}$ and $g(p_1) \in D$,
        \item $g(q) \in D(q_1, 1/2)$ and $\abs{g'} < 1/2$ of $F_1$.
        \item $g(C_2) \subseteq S_2$, $g(F_1) \subseteq S_1$, and $g^2(F_1) \subseteq S_2$,
        \item $g(F_1) \subseteq V_1$ with $\dist(g(F_1), \partial V_1) \ge \lambda_1$,
        \item $g(F_1) \subseteq C_2^+$ with $\dist(g(F_1), \partial C_2^+) > 0$,
        \item the function $g$ is injective on $V_1$ and maps uniformly accessible subsets of $V_1$ to uniformly accessible ones,
        \item for every $k \ge 1$ the set $g^2(F_k)$ is closed and uniformly accessible.
    \end{enumerate} 

    Since $\dist(V_1, \partial E_1^- \cup \partial E_1^+) = \lambda_1$ and $\abs{h_1'}$ is a positive constant on $E_1^-$ and on $E_1^+$, 
    item (i) can be guaranteed by using Cauchy estimates. In particular, since $\abs{h_1'} < 1/2$ we can also guarantee $\abs{g'} < 1/2$ on $E_1 \cap V_1$.
    Since we chose $\varepsilon_1 < 1/4$, items (ii) and (iii) immediately follow.
    
    Next we apply Lemma~\ref{unbounded_injective_uniff_access} twice; 
    once to the open set $\mathring{E}_1^-$, the translation $\tau$, and the constant $\lambda_1$,
    and once the the open set $\mathring{E}_1^+$, the linear map $L_1$, and the constant $\lambda_1$.
    Note that this is possible, since the sets $E_1^-$ and $E_1^+$ are Arakelian, so the sets $\mathring{E}_1^-$ and $\mathring{E}_1^+$ are simply connected. 
    We take $\varepsilon_1$ small enough so that the conclusions of both lemmas hold.
    It then immediately follows that a function $g$ satisfying $\norm{g- h_1}_{\Pi_1} < 2\varepsilon_1$ will be injective on $V_1$ and will map uniformly accessible subsets of $V_1$ to uniformly accessible ones, thus showing item (vii).
    It also follows that $g(C_2) \subseteq S_2$ and $g(F_1) \subseteq C_2^+$ with $\dist(g(F_1), \partial C_2^+) > 0$, so item (vi) holds.
    Since $g(F_1) \subseteq C_2^+ \subseteq S_1 \cap V_1$, $\dist(C_2^+, \partial V_1) \ge 2\lambda_1$, and $g^2(F_1) \subseteq g(C_2) \subseteq S_2$, items (iv) and (v) hold.
    Since $F_k \cup g(F_k) \subseteq V_1$ for all $k \ge 1$, item (viii) is also immediate.
    Note, that the bounds on $\abs{g'}$ do not depend on the function $g$.
    
    Finally, we use the Arakelian approximation theorem to obtain an entire function $f_1 \in \mathcal{O}(\C)$ such that $\norm{f_1 - h_1}_{\Pi_1} < \varepsilon_1$.
    With that, items (1) through (8) are satisfied for the case $n=1$.

    \subsection*{The inductive step:} Suppose, we have already completed the $n$-th step of the inductive proces.
    Item (4) of the inductive hypothesis implies $f^n_n(C_{n+1}) \subseteq S_{n+1}$,
    while items (5) and (7) imply that the function $f_n^n$ is injective on a neighbourhood of $C_{n+1}$.
    Since the sets $\mathring{C}_{n+1}^-$ and $\mathring{C}_{n+1}^+$ are disjoint and simply connected, 
    it follows that $f_n^n(\mathring{C}_{n+1}^-)$ and $f_n^n(\mathring{C}_{n+1}^+)$ are disjoint simply connected open subsets of $S_{n+1}$.
    Recall that $C_{n+2} \subseteq \mathring{C}_{n+1}^-$ with $\dist(C_{n+2}, \partial \mathring{C}_{n+1}^-) > 0$ by definition,
    while $f_n^{N_n}(F_n) \subseteq \mathring{C}_{n+1}^+$ with $\dist(f_n^{N_n}(F_n), \mathring{C}_{n+1}^+) > 0$ by item (6) of the induction hypothesis.
    Since $\abs{(f_n^n)'}$ is uniformly bounded from below on $C_{n+1}$ by items (3) and (5) of the induction hypothesis,
    an application of Koebe's 1/4-theorem yields that $f_n^n(C_{n+2}) \subseteq f_n^n(\mathring{C}_{n+1}^-)$ with $\dist(f_n^n(C_{n+2}), f_n^n(\mathring{C}_{n+1}^-)) > 0$ and $f_n^{N_n+n}(F_n) \subseteq f_n^n(\mathring{C}_{n+1}^+)$ with $\dist(f_n^{N_n+n}(F_n), f_n^n(\mathring{C}_{n+1}^+)) > 0$.
    Items (5) and (7) of the induction hypothesis also imply that the function $f_n^{N_n+n}$ is injective on $F_n$, hence $f_n^{N_n+n}(p_{n+1})$ and $f_n^{N_n+n}(F_{n+1})$ are disjoint subsets of $f_n^{N_n+n}(F_n)$.
    Finally, item (8) of the induction hypothesis implies that the set $f_n^{N_n+n}(F_{n+1})$ is uniformly accessible.

    We apply the corollary of Proposition~\ref{Arakelian_so_gosti} to the sets $f_n^n(C_{n+2})$ and $f_n^n(\mathring{C}_{n+1}^-)$ to obtain an Arakelian set $E_{n+1}^-$ such that $f_n^n(C_{n+2}) \subseteq E_{n+1}^- \subseteq f_n^n(\mathring{C}_{n+1}^-)$ and $\dist(f_n^n(C_{n+2}), \partial E_{n+1}^-) > 0$.
    We also apply Lemma~\ref{Strip_Arakelian} to the point $f_n^{N_n+n}(p_{n+1})$ and sets $f_n^{N_n+n}(F_{n+1})$ and $f_n^n(\mathring{C}_{n+1}^+)$ to obtain an Arakelian set $E_{n+1}^+$ such that $f_n^{N_n+n}(F_{n+1}) \subseteq E_{n+1}^+ \subseteq f_n^n(\mathring{C}_{n+1}^+)$, $\dist(f_n^{N_n+n}(p_{n+1}), \partial E_{n+1}^+) > 0$ and $f_n^{N_n+n}(p_{n+1}) \in \C \setminus \partial E_{n+1}^+$. 
    Observe that the Arakelian sets $E_{n+1}^-$ and $E_{n+1}^+$ are disjoint. We then denote 
    \[\lambda_{n+1} = \min \set{\dist(f_n^n(C_{n+2}), \partial E_{n+1}^-)/3, \, \dist(f_n^{N_n+n}(p_{n+1}), \partial E_{n+1}^+)/3, \, \lambda_n}\] 
    and define the sets 
    \[A_{n+1} = \set{z \in E_{n+1}^- \cup E_{n+1}^+ \mid \dist(z, \partial E_{n+1}^- \cup \partial E_{n+1}^+) \ge \lambda_{n+1}} \quad \text{and} \quad V_{n+1} = V_n \cup \mathring{A}_{n+1}.\]
    Note $\dist(A_{n+1}, \partial E_{n+1}^- \cup \partial E_{n+1}^+) = \lambda_{n+1}$ and $\dist(f_n^n(C_{n+2}) \cup f_n^{N_n+n}(F_{n+1}), \partial A_{n+1}) \ge 2\lambda_{n+1}$.
    Since $f_n^{N_n+n}(p_{n+1}) \in f_n^n(\mathring{C}_{n+1}^+)$, we can choose a closed disc $B_{n+1} \subseteq f_n^n(\mathring{C}_{n+1}^+) \setminus E_{n+1}^+$ centered in $f^n_n(p_{n+1})$.
    Finally, we define the set 
    \[\Pi_{n+1} = P_n \cup E_{n+1}^- \cup E_{n+1}^+ \cup B_{n+1}.\]
    Since $P_n$, $E_{n+1}^-$, $E_{n+1}^+$, and $B_{n+1}$ are pairwise disjoint Arakelian sets, their union $\Pi_{n+1}$ is also Arakelian.

    By item (3) of the induction hypothesis, $|f_n'|$ is bounded above on $V_n$ by some constant $M$.
    We choose a linear map $L_{n+1}$ such that $L_{n+1}(E_{n+1}^+) \subseteq C_{n+2}^+$ with $\dist(L_{n+1}(E_{n+1}^+), \partial C_{n+2}^+) > 0$, ${L_{n+1}(f_n^{N_n + n}(q)) = q_{n+1}}$ and that $\abs{L_{n+1}'} < M^{-N_n - n}/2^{n+1}$.
    This is possible, since $E_{n+1}^+ \subseteq S_{n+1}$.
    We then define the function $h_{n+1} \colon \Pi_{n+1} \to \C$ by setting
    \[h_{n+1}(z) = \begin{cases} f_n;& z \in P_n \\ z+3;& z \in E_{n+1}^- \\ L_{n+1}(z);& z \in E_{n+1}^+ \\ -3;& z \in B_{n+1} \end{cases}.\]
    Note, that the function $h_{n+1}$ extends holomorphically to an open neighbourhood of $\Pi_{n+1}$, that is $h_{n+1} \in \mathcal{O}(\Pi_{n+1})$.
    
    We now chose $0 < \varepsilon_{n+1} < \varepsilon_n/2$ small enough, so that if an entire function $g \in \mathcal{O}(\C)$ satisfies $\norm{g- h_{n+1}}_{\Pi_{n+1}} < 2\varepsilon_{n+1}$, the following holds:

    \begin{enumerate}[\indent (i)]
        \item $\abs{g'}$ is uniformly bounded above and below on $V_{n+1}$ by some positive constants,
        \item $g^{N_{n+1}}(p_{n+1}) \in D$,
        \item $g^{N_{n+1}}(q) \in D(q_{n+1}, 1/2^{n+1})$ and $\abs{\left(g^{N_{n+1}}\right)'} < 1/2^{n+1}$ on $F_{n+1}$,
        \item $g^k(C_{n+2}) \subseteq S_{k+1}$ for all $0 \le k \le n+1$ and $g^{N_m + k}(F_{n+1}) \subseteq S_{k+1}$ for all $0 \le k \le m \le n+1$,
        \item $g^k(C_{n+2}) \subseteq V_{n+1}$ with $\dist(g^k(C_{n+2}), \partial V_{n+1}) \ge \lambda_{n+1}$ for all $0 \le k \le n$ and $g^{k}(F_{n+1}) \subseteq V_{n+1}$ with $\dist(g^{k}(F_{n+1}), \partial V_{n+1}) \ge \lambda_{n+1}$ for all $0 \le k \le N_{n+1}+n$,
        \item $g^{N_{n+1}}(F_{n+1}) \subseteq C_{n+2}^+$ with $\dist(g^{N_{n+1}}(F_{n+1}), \partial C_{n+2}^+) > 0$,
        \item the function $g$ is injective on $V_{n+1}$ and maps uniformly accessible subsets of $V_{n+1}$ to uniformly accessible ones,
        \item for every $k \ge n+1$ the set $g^{N_{n+1} + n + 1}(F_k)$ is closed and uniformly accessible.
    \end{enumerate} 
    Since $\Pi_n \subseteq P_n$ and $\varepsilon_{n+1} < \varepsilon_n/2$, the function $g$ will satisfy $\norm{g- f_n}_{\Pi_{n}} < \varepsilon_{n}$,
    so by the induction hypothesis, it will enjoy properties (3) through (8).

    For item (i), notice that item (3) of the induction hypothesis implies that $\abs{g'}$ is uniformly bounded above and below on $V_{n}$,
    while the uniform bound on $\mathring{A}_{n+1}$ can be guaranteed by using Cauchy estimates, since $\dist(A_{n+1}, \partial E_{n+1}) = \lambda_{n+1}$ and $\abs{h_{n+1}'}$ is a positive constant on $E_{n+1}^-$ and on $E_{n+1}^+$.
    In particular, since $\abs{h_{n+1}'} < 2^{-(n+1)}M^{-(N_n + n)}$, we can also guarantee $\abs{g'} < M^{-N_n - n}/2^{n+1}$ on $E_{n+1}^+ \cap \mathring{A}_{n+1}$.

    Notice that $\abs{h_{n+1}'}$ is uniformly bounded above on $V_n$ by item (3) of the induction hypothesis and on $\mathring{A}_{n+1}$ by construction.
    Furthermore, since $C_{n+2} \subseteq C_{n+1}$, $\set{p_{n+1}} \cup F_{n+1} \subseteq F_n$, and $\lambda_{n+1} \le \lambda_n$, we have
    \begin{align*}
        \widehat{C}_{n+2} &= \bigcup_{k=0}^{n-1}h^k_{n+1}(C_{n+2}) \subseteq V_{n+1} \subseteq \mathring{\Pi}_{n+1}\\
        \widehat{F}_{n+1} &= \bigcup_{k=0}^{N_{n} + n}h^k_{n+1}(F_{n+1}) \subseteq V_{n+1} \subseteq \mathring{\Pi}_{n+1}\\
        \widehat{p}_{n+1} &= \bigcup_{k=0}^{N_{n} + n-1}h^k_{n+1}(p_{n+1}) \subseteq V_{n+1} \subseteq \mathring{\Pi}_{n+1}\\
    \end{align*}
    with $\dist(\widehat{F}_{n+1} \cup \widehat{p}_{n+1} \cup \widehat{C}_{n+1}, V_{n+1}) \ge \lambda_{n+1}$ by item (5) of the induction hypothesis.
    Thus, we may apply Lemma~\ref{approxIter} to each of the above sets separately and, by choosing $\varepsilon_{n+1}$ sufficiently small, we can guarantee that
    \begin{alignat*}{2}
        \norm{g^n - f_n^n}_{C_{n+2}} &< \lambda_{n+1} &\qquad \norm{g^{N_n + n} - f_n^{N_n + n}}_{F_{n+1}} &< \lambda_{n+1} \\
        g^{N_n + n}(p_{n+1}) &\in B_{n+1} &\qquad \abs{g^{N_{n+1}}(q) - h_{n+1}^{N_{n+1}}(q)} &< \frac{1}{2^{n+1}}.
    \end{alignat*}
    The top two estimates, together with $\dist(f^n_n(C_{n+2}) \cup f_n^{N_n + n}(F_{n+1}), \partial A_{n+1}) \ge 2\lambda_{n+1}$ 
    imply $g^n(C_{n+2}) \subseteq E_{n+1}^- \cap \mathring{A}_{n+1}$ with $\dist(g^n(C_{n+2}), \partial A_{n+1}) \ge \lambda_{n+1}$ and $g^{N_n + n}(F_{n+1}) \subseteq E_{n+1}^+ \cap \mathring{A}_{n+1}$ with $\dist(g^{N_n + n}(F_{n+1}), \partial A_{n+1}) \ge \lambda_{n+1}$ respectively.
    Since $h_{n+1} = -3$ on $B_{n+1}$ and $\varepsilon_{n+1} < 1/4$ by item (1) of the induction hypothesis, the bottom left estimate implies $g^{N_{n+1}}(p_{n+1}) \in \D$, showing item (ii).
    Since $h_{n+1}^{N_{n+1}}(q) = L_{n+1}(f_n^{N_n+n}(q)) = q_{n+1}$, the bottom right estimate implies $g^{N_{n+1}}(q) \in D(q_{n+1}, 1/2^{n+1})$.
    Lastly, by item (3) of the induction hypothesis, $|g'|$ is also bounded above on $V_n$ by $M$. 
    Since $g^{N_n + n}(F_{n+1}) \subseteq E_{n+1}^+ \cap \mathring{A}_{n+1}$, we then have the estimate 
    \[\abs{\left(g^{N_{n+1}}\right)'(z)} = \abs{g'\left(g^{N_n+n}(z)\right)}\abs{\left(g^{N_n+n}\right)'(z)} < \frac{M^{-N_n - n}}{2^{n+1}} M^{N_n+n} = \frac{1}{2^{n+1}}\]
    for all $z \in F_{n+1}$, thus showing item (iii).

    Next, we apply Lemma~\ref{unbounded_injective_uniff_access} twice; 
    once to the open set $\mathring{E}_{n+1}^-$, the translation $\tau$ and the constant $\lambda_{n+1}$,
    and once the the open set $\mathring{E}_{n+1}^+$, the linear map $L_{n+1}$ and the constant $\lambda_{n+1}$. 
    Note that this is possible, since the sets $E_{n+1}^+$ and $E_{n+1}^-$ are Arakelian, so the sets $\mathring{E}_{n+1}^+$ and $\mathring{E}_{n+1}^-$ are simply connected.
    We take $\varepsilon_{n+1}$ small enough so that the conclusions of both lemmas hold.
    It then immediately follows that a function $g$ satisfying $\norm{g- h_{n+1}}_{\Pi_{n+1}} < 2\varepsilon_{n+1}$ will be injective on $\mathring{A}_{n+1}$ and will map uniformly accessible subsets of $\mathring{A}_{n+1}$ to uniformly accessible ones.
    By item (7) of the induction hypothesis, the same properties hold on $V_n$, so item (vii) holds.
    It also follows that $g^{n+1}(C_{n+2}) \subseteq S_{n+2}$ and $g^{N_{n+1}}(F_{n+1}) \subseteq C_{n+2}^+$ with $\dist(g^{N_{n+1}}(F_{n+1}), \partial C_{n+2}^+) > 0$, so item (vi) holds.
    By item (4) of the induction hypothesis. we then have $g^k(C_{n+2}) \subseteq S_{k+1}$ for all $0 \le k \le n+1$, 
    and since $g^{N_{n+1}}(F_{n+1}) \subseteq C_{n+2}$, we also have $g^{N_m + k}(F_{n+1}) \subseteq S_{k+1}$ for all $0 \le k \le m \le n+1$, so item (iv) holds.
    Similarly, by item (5) and the estimates following from Lemma~\ref{approxIter} we have $g^k(C_{n+2}) \subseteq V_{n+1}$ with $\dist(g^k(C_{n+2}), \partial V_{n+1}) \ge \lambda_{n+1}$ for all $0 \le k \le n$,
    and since $g^{N_{n+1}}(F_{n+1}) \subseteq C_{n+2}$, we also have $g^{k}(F_{n+1}) \subseteq V_{n+1}$ with $\dist(g^{k}(F_{n+1}), \partial V_{n+1}) \ge \lambda_{n+1}$ for all $0 \le k \le N_{n+1}+n$, so item (v) also holds.
    Since $g^{N_n +n}(F_k)$ is closed and uniformly accessible by item (8) of the induction hypothesis, and remains in $V_{n+1}$ under $n+1$ iteration of $g$ for all $k \ge n+1$, item (viii) is also immediate.
    Note, that the bounds on $\abs{g'}$ do not depend on the function $g$.
    
    Finally, we use the Arakelian approximation theorem to obtain an entire function $f_{n+1} \in \mathcal{O}(\C)$ such that $\norm{f_{n+1} - h_{n+1}}_{\Pi_{n+1}} < \varepsilon_{n+1}$.
    From the above construction, it is clear that items (1) through (8) are satisfied for the case $n + 1$.  
    
    \subsection*{Verification:} Condition (1) implies that for every $n \ge 1$ the functions $(f_k)_{k \ge n}$ form a Cauchy sequence on $\Pi_n$, and by condition (2) we have $\bigcup_{n=1}\Pi_n = \C$.
    Thus, the sequence $(f_n)_{n \ge 1}$ converges uniformly on compact sets to some entire function $f \in \mathcal{O}(\C)$.
    Given $n \ge 1$ we then have $f = f_n + \sum_{k=n}^\infty f_{k+1}-f_k$, which yields the estimate 
    \[\norm{f - f_n}_{\Pi_n} \le \sum_{k=n}^\infty \norm{f_{k+1}-f_k}_{\Pi_n} < \sum_{k=n}^\infty \varepsilon_{k+1} < \varepsilon_n \sum_{k=n}^\infty 2^{k+1-n} < \varepsilon_n,\]
    so by the induction hypothesis, the function $f$ enjoys the same properties as $f_n$.
    In particular, we have $f^{N_m + k}(F) \subseteq S_{k+1}$ for all $0 \le k \le m$, $f(D) \subseteq \mathring{D}$, and $f^{N_n}(p_n) \in D$ for all $n \ge 1$, thus proving items (a) and (c).
    Since $f^{N_n}(q) \in D(q_n, 1/2^n)$ and $\lim_{n \to \infty} q_n = -2/3$, we have $\lim_{n \to \infty} f^{N_n}(q) = -2/3$,
    and since $\abs{(f^{N_n})'} < 1/2^n$, the estimate $\abs{f^{N_n}(p) - f^{N_n}(q)} \le \abs{p-q}/2^n$ implies that $\lim_{n \to \infty}f^{N_n} = -2/3$ uniformly on compact subsets of $F$, proving item (b).
    Since $f^k(F) \subseteq V_n$ for all $0 \le k \le N_n + n - 1$, the function $f^{N_n + n}$ is injective on $F$ for all $n \ge 1$, so in particular item (d) holds.
    Finally, since $f^{N_n + n}$ is injective on a neighbourhood of $F$ and $\abs{(f^{N_n + n})'}$ is uniformly bounded from below on this neighbourhood, 
    the inverse function $(f^{N_n + n})^{-1}$ is Lipschitz on $f^{N_n + n}(F)$. Thus $f^{N_n + n}$ maps unbounded subsets of $F$ to unbounded ones for all $n \ge 1$, which implies item (e).
\end{proof}

\section{Lifting wandering domains} \label{sec4}

In this section, we construct wandering domains by lifting bounded wandering domains with prescribed geometry via the exponential map. 
To this end, we will need a modification of the classical Runge approximation theorem that allows us to approximate a nonvanishing holomorphic function 
defined on a neighbourhood of a compact set by a nonvanishing entire function.

\begin{lemma}\label{Runge_omitted_value}
    Let $K$ be a compact set without holes, $h \in \mathcal{O}(K)$ a nonvanishing holomorphic function, and $\varepsilon > 0$.
    Then there exists a nonvanishing entire function $f \in \mathcal{O}(\C)$ such that $\norm{f-h}_K < \varepsilon$.
\end{lemma}

\begin{proof}
    Let $U \subseteq \C$ be a simply connected open neighbourhood of $K$ on which $h$ is holomorphic.
    Such a neighbourhood exists, since the set $K$ is Arakelian, see Theorem~\ref{characterization}.
    Since the function $h$ is nonvanishing and $U$ is simply connected, there exists a holomorphic function $\widetilde{h} \in \mathcal{O}(U)$ such that $h = e^{\widetilde{h}}$.
    Denote $M = \max_{z \in K}\abs{h(z)} > 0$. By the continuity of the exponential map, there exists $\delta > 0$ such that $\abs{z} < \delta$ implies $\abs{e^z-1} < \frac{\varepsilon}{M}$.
    We now use the classical Runge approximation theorem to obtain an entire function $\widetilde{f} \in \mathcal{O}(\C)$ such that $\norm{\widetilde{f}-\widetilde{h}}_K < \delta$ and define $f = e^{\widetilde{f}}$.
    Clearly $f$ is a nonvanishing entire function and a quick computation shows
    \[\abs{f(z)-h(z)} = \abs{e^{\widetilde{f}(z)} - e^{\widetilde{h}(z)}} = \abs{e^{\widetilde{h}(z)}}\abs{e^{\widetilde{f}(z)-\widetilde{h}(z)}-1} < M \frac{\varepsilon}{M} = \varepsilon\]
    for $z \in K$, taking into account that $\abs{\widetilde{f}(z)-\widetilde{h}(z)} < \delta$. This concludes the proof.
\end{proof}

We are now ready to prove Theorem~\ref{liftWD}.

\begin{proof}[Proof of Theorem~\ref{liftWD}]
    Since $\Omega_0$ is a bounded, connected, regular open set whose closure has connected complement, we know that the set can be realized as the wandering domain of some entire function $g \in \mathcal{O}(\C)$ whose iterates are univalent on the wandering domain.
    The key idea is to construct the function $g$ so that it is nonvanishing.
    To this end, we follow the proofs of \cite[Theorems 4 and 5]{GeometryOfWD} or Theorems~\ref{stripWDescaping} and~\ref{stripWDoscillating},
    but use Lemma~\ref{Runge_omitted_value} instead of the classical Runge or Arakelian approximation theorems. 
    At each step of the inductive construction, we obtain a nonvanishing entire function $g_{n+1}$ by approximating the function $g_n$ from the previous step, which is nonvanishing by the induction hypothesis, and a linear map that can always be chosen to map away from the origin. 
    We thus obtain a sequence of nonvanishing entire functions $(g_{n})_{n \ge 1}$ and define $g$ as the uniform limit of this sequence.
    Since the functions $g_n$ are all nonvanishing and the function $g$ is non-constant by construction, it follows from Hurwitz's theorem that the function $g$ is nonvanishing as well.
    We can also modify the construction of $g$ to make sure that the forward orbit of $g(\Omega_0)$ lies entirely in the right half-plane. 

    Next consider the restriction $g \vert_{\C \setminus \set{0}}$.
    There exists an entire function $f \in \mathcal{O}(\C)$ such that $g \circ \exp = \exp \circ f$, and it follows by induction that $g^n \circ \exp = \exp \circ f^n$ for all $n \ge 1$.   
    A result of Bergweiler, see~\cite{BergweilerExpLift}, guarantees that the preimage of a Fatou component of $g$ under the exponential map belongs to the Fatou set of $f$.
    In particular, the connected components of the preimage are Fatou components of the function $f$.
    Now consider the set $\Omega = \exp^{-1}(\Omega_0)$ and note it is connected. 
    By the previous observation, it is a Fatou component of the function $f$, and we claim it is also a wandering domain. 
    Indeed, suppose that $f^n(\Omega) = f^m(\Omega)$ for some $n,m \ge 1$. Applying the exponential map to both sides and using the relation $g^n \circ \exp = \exp \circ f^n$, we obtain $g^n(\Omega_0) = g^m(\Omega_0)$, and since $\Omega_0$ is a wandering domain of $g$, this implies $n=m$.
    Furthermore, since the orbit of $g(\Omega_0)$ lies in the right half-plane, where the appropriate branch of $\log$ is well defined, the wandering domain $\Omega$ is escaping if and only if $\Omega_0$ is escaping, and oscillating if and only if $\Omega_0$ is oscillating.

    It is obvious that the set $\Omega$ is unbounded and contains a half-plane. 
    To see that the sets $f^n(\Omega)$ for $n \ge 1$ are bounded, note that $\exp(f^n(\Omega)) = g^n(\Omega_0)$, so $f^n(\Omega)$ is a connected component of the set $\exp^{-1}(g^n(\Omega_0))$. 
    But since $g^n(\Omega_0)$ is a bounded set in the right half-plane, all the connected components of $\exp^{-1}(g^n(\Omega_0))$ are bounded.
    Finally, note that the exponential map is $\infty$-to-$1$ on $\Omega$ and $1$-to-$1$ on $f^n(\Omega)$ for $n \ge 1$.
    Since $f \vert_\Omega = \log \circ g \circ \exp$ for the appropriate choice of logarithm, the function $f \vert_\Omega$ is $\infty$-to-$1$ while the function $f \vert_{f^n(\Omega)}$ is injective for all $n \ge 1$.     
\end{proof}

\begin{remark}
    Since $g \circ \exp = \exp \circ f$, there exists some $k \in \Z$ such that $f(z + 2\pi i) = f(z) + 2\pi i k$, see \cite[Lemma 2]{BergweilerExpLift}.
    If we pick $z_0 \in \Omega$, we have $z_0 + 2\pi i \in \Omega$ and the invariance property of Fatou components implies $f(z_0), f(z_0 + 2\pi i) \in f(\Omega)$.
    By the above construction the set $f(\Omega)$ has horizontal width at most $\pi$, so $k = 0$ and the function $f$ is $2\pi i$-periodic.
    In particular, all preimages of $\exp^{-1}(g^n(\Omega_0))$ are mapped to the same Fatou component $f^{n+1}(\Omega)$, see Figure~\ref{dynamics_lift}. 
\end{remark}

\begin{figure}[ht]
    \centering
    \includegraphics[width=0.65\textwidth]{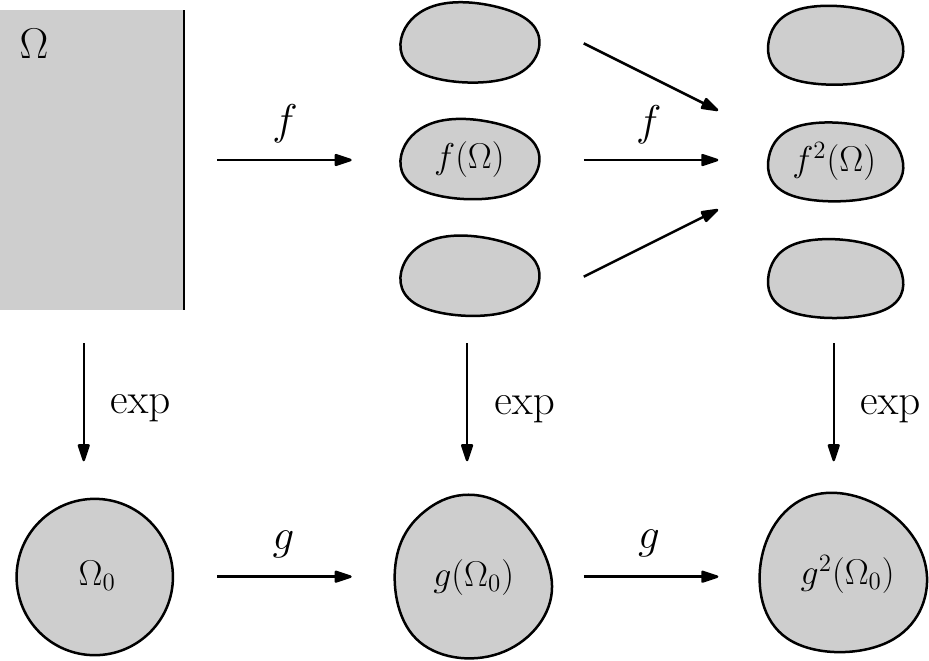}
    \caption{The dynamics of the maps $g$ and $f$ in the case of an escaping wandering domain. Note the the function $f$ is $2\pi i$-periodic.}
    \label{dynamics_lift} 
\end{figure}

\section{Wandering domains with complement of small area} \label{sec5}

In this short section, we give a direct proof of Theorem~\ref{hornWD}. 
As mentioned in the introduction, this result is an immediate consequence of Theorems~\ref{approxWD_escaping} and~\ref{approxWD_oscillating},
however proving it directly will demonstrate the main idea behind the proofs of those theorems. 

\begin{proof}[Proof of Theorem~\ref{hornWD}]
    The idea is to construct an Arakelian set $\Pi$ and a holomorphic function $h \in \mathcal{O}(\Pi)$ that will model the desired dynamical behavior.
    The function $f$ is then obtained by applying Proposition~\ref{PosplositevArakelian} to the function $h$ and a suitable locally constant error function $\varepsilon$.

    For $k \ge 1$ we define the function $g_k(x) = \frac{\delta}{2}x^{-(k+1)}$ and denote $r_k = \min\set{\frac{1}{4}, \frac{1}{2}\, g_k\left(k+\frac{1}{2}\right)}$.
    This allows us to construct the following sets:
    \begin{align*}
        E &= (-\infty, 1] \times \R \cup \set{(x,y) \in [1, \infty) \times \R \mid g_1(x) \le \abs{y}} \\
        \gamma_k &= \set{k} \times [-g_k(k),g_k(k)] \cup \set{(x,y) \in [k, \infty) \times \R \mid g_k(x) = \abs{y}} \text{ for } k \ge 2 \\
        D_k &= D\left(k + 1/2, r_k\right) \text{ for } k \ge 1.
    \end{align*}
    We then define the set
    \[\Pi = E \cup \bigcup_{k=2}^\infty \gamma_k \cup \bigcup_{k=1}^\infty D_k,\]
    see Figure~\ref{fig_rog}.
    Note that since the sets $E$, $(\gamma_k)_{k \ge 2}$, and $(D_k)_{k \ge 1}$ are pairwise disjoint, locally finite, and Arakelian, it follows from Proposition~\ref{locallyfiniteArakelian} that the set $\Pi$ is also Arakelian.
    Next we define the locally constant functions $h \in \mathcal{O}(\Pi)$ and $\varepsilon \colon \Pi \to (0,\infty)$ by 
    \[h(z) = \begin{cases} 
        \frac{5}{2}; & z \in E \\
        \frac{3}{2}; & z \in \gamma_k \text{ for } k \ge 2 \\ 
        \frac{3}{2}; & z \in D_1\\ 
        k + \frac{3}{2}; & z \in D_k \text{ for } k \ge 2 
    \end{cases} 
    \quad \text{and} \quad 
    \varepsilon(z) = \begin{cases}
        \frac{r_2}{2}; & z \in E \\
        \frac{r_1}{2}; & z \in \gamma_k \text{ for } k \ge 2 \\ 
        \frac{r_1}{2}; & z \in D_1\\ 
        \frac{r_{k+1}}{2}; & z \in D_k \text{ for } k \ge 2 
    \end{cases}.\]
    Applying Proposition~\ref{PosplositevArakelian} we obtain an entire function $f \in \mathcal{O}(\C)$ such that 
    $\abs{f(z)-h(z)} < \varepsilon(z)$ for all $z \in \Pi$.
    In particular, this implies
    $f(D_1) \subseteq \mathring{D}_1$ and $f(E) \subseteq \mathring{D}_2$ as well as $f(\gamma_k) \subseteq \mathring{D}_1$ and $f(D_k) \subseteq \mathring{D}_{k+1}$ for all $k \ge 2$.

    Observe that the function $f$ acts as a contraction on the set $D_1$, hence $D_1$ is contained in some attracting basin $\mathcal{A}$,
    and since $f(\gamma_k) \subseteq D_1$ for all $k \ge 2$, we have $\bigcup_{k \ge 2} \gamma_k \subseteq f^{-1}(\mathcal{A})$.
    On the other hand, since $f^n(E) \subseteq D_{n+1}$ for all $n \ge 1$, the points in $E$ escape uniformly to infinity under iteration, 
    hence $\mathring{E}$ is contained in some other Fatou component $\Omega_1$. 
    By the same argument, the sets $\mathring{D}_k$, with $k \ge 2$, are contained in their respective Fatou components $\Omega_k$.
    In particular, since $f^n(E) \subseteq \mathring{D}_{n+1}$, we have $f^n(\Omega_1) \subseteq \Omega_{n+1}$ for all $n \ge 1$.

    We claim $\Omega_n \cap \Omega_m = \emptyset$ for all $n \neq m$. 
    If not, there exists a pair of integers, say $n > m$, such that the set $\Omega_n \cup \Omega_m$ is connected. 
    Let $\eta \subseteq \Omega_n \cup \Omega_m$ be a curve connecting a point in $D_n$ to a point in $D_m$ if $m \ge 2$ or $E$ if $m = 1$.
    By construction, the curve $\eta$ must intersect the set $\gamma_n$ and thus contains a point that gets sent to the attracting basin $\mathcal{A}$.
    This, however, contradicts the fact that all points in $\Omega_n \cup \Omega_m$ escape uniformly to infinity under iteration.

    Since $f^n(\Omega_1) \subseteq \Omega_{n+1}$ for all $n \ge 1$, the above observation implies that the set $\Omega = \Omega_1$ is a wandering domain.
    To conclude the proof, note that since $\mathring{E} \subseteq \Omega$, we have the estimate
    \[\Area(\C \setminus \Omega) \le \Area(\C \setminus \mathring{E}) = 2 \int_1^\infty g_1(x) \,dx = \delta \int_1^\infty \frac{1}{x^2} \, dx = \delta. \qedhere\]
\end{proof}

\begin{figure}[ht]
    \centering
    \includegraphics[width=0.9\textwidth]{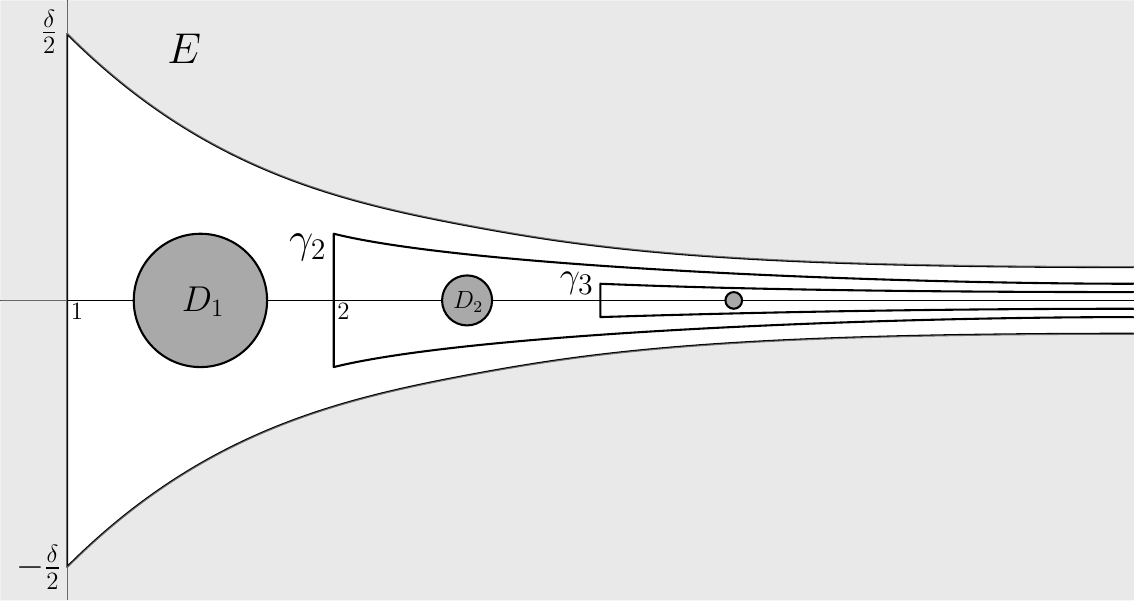}
    \caption{The construction of the set $\Pi$. The discs $(D_k)_{k \ge 2}$ facilitate uniform convergence to infinity, 
    while the curves $(\gamma_k)_{k \ge 2}$ ensure that the discs belong to distinct Fatou components.}
    \label{fig_rog} 
\end{figure} 

\begin{remark}
    Notice that the function $f$ constructed above has an unbounded Fatou component; for example, the preimage of $\mathcal{A}$ containing the curve $\gamma_2$.
    This implies that all the Fatou components of the function $f$ are simply connected, see~\cite[Theorem 3.1.6]{HolomorphicDynamics}.
    In particular, the constructed wandering domain $\Omega$ is also simply connected. 
    This happens in general when one uses unbounded curves to separate wandering sets in order to ensure they belong to different Fatou components.
\end{remark}

\begin{remark}
    Notice that by construction $\mathring{E} \subseteq \Omega$. 
    On the other hand, since the sets $\Omega$ and $\gamma_2$ are disjoint, we also know that
    \[\Omega \subseteq (-\infty, 2] \times \R \cup \set{(x,y) \in [2, \infty) \times \R \mid g_2(x) \le \abs{y}}.\]
    In particular, if we were to add another separating set $\Gamma \subseteq \C \setminus E$, arbitrarily close to the boundary of $E$, and made sure that $f(\Gamma) \subseteq \mathcal{A}$,
    the shape of the wandering domain $\Omega$ could be made arbitrarily close to that of the set $E$.
    This is the main idea behind the proof of Theorems~\ref{approxWD_escaping} and~\ref{approxWD_oscillating}.
\end{remark}

We conclude this section by showing how a small modification of the above construction allows us to construct a periodic Fatou component with complement of arbitrarily small area.

\begin{corollary}
    For each $\delta > 0$ and $p \in \N$ there exists some entire function $f \in \mathcal{O}(\C)$ with a periodic Fatou component $\Omega$ of period $p$ such that $\Area(\C \setminus \Omega) < \delta$.
\end{corollary}

\begin{proof}
    For the case $p = 1$, we can give an explicit example. 
    Recall that the Fatou set of the exponential family $z \mapsto \lambda e^z$ with $0 < \lambda < 1/e$ consists of just one attracting basin $\Omega$, see~\cite[Proposition 3.2.1]{HolomorphicDynamics}. 
    In particular, $\C \setminus \Omega$ is the Julia set of the function which has area zero. 
    
    For the case $p \ge 2$, we proceed similarly as in the proof of Theorem~\ref{hornWD} above. 
    Define the function $g(x) = \frac{\delta}{2}x^{-2}$ and denote $r = \min\set{\frac{1}{4}, \, \frac{1}{2}g\left(k + \frac{1}{2}\right) \mid 2 \le k \le p}$.
    This allows us to construct the following sets: 
    \begin{align*}
        E &= (-\infty, 1] \times \R \cup \set{(x,y) \in [1, \infty) \times \R \middle \vert g(x) \le \abs{y}} \\
        D_k &= D\left(k + 1/2, r\right) \text{ for } 2 \le k \le p.
    \end{align*}
    We then define the set $\Pi = E \cup \bigcup_{k=2}^{p} D_k$ and observe that it is Arakelian. 
    Take $\varepsilon = r/2$ and define the locally constant function $h \in \mathcal{O}(\Pi)$ by 
    \[h(z) = \begin{cases} 
        \frac{5}{2}; & z \in E \\
        k + \frac{3}{2}; & z \in D_k \text{ for } 2 \le k \le p-1 \\
        0; & z \in D_p.
    \end{cases}\] 
    Applying Proposition~\ref{PosplositevArakelian} we obtain an entire function $f \in \mathcal{O}(\C)$ such that 
    $\abs{f(z)-h(z)} < \varepsilon$ for all $z \in \Pi$.
    In particular, this implies $f(E) \subseteq \mathring{D}_2$, $f(D_k) \subseteq \mathring{D}_{k+1}$ for $2 \le k \le p-1$, and $f(D_{p-1}) \subseteq \mathring{D}(0, 1/4) \subseteq \mathring{E}$.
    It follows that the function $f^p$ acts as a contraction on each of the sets $E$ and $D_2, \ldots, D_{p}$,
    hence each set is contained in a distinct fixed attracting basin $\Omega_1$ and $\Omega_2, \ldots, \Omega_{p}$ respectively.
    It is clear from the construction that these form an attracting cycle of period $p$, 
    and a similar computation as in the proof of Corollary~\ref{hornWD} for $\Omega = \Omega_1$ shows that $\Area(\C \setminus \Omega) < \delta$.
\end{proof}

\section{Wandering domains of approximate shape} \label{sec6}

The goal of this final section is to prove Theorems~\ref{approxWD_escaping},~\ref{approxWD_oscillating}, and~\ref{combination}.
As already outlined in the previous section, the main idea of the proofs is to use a separating set $\Gamma$, 
that is sent to an attracting basin, as an upper bound for the size of the wandering domains.
In the construction of the previous section, $\Gamma$ can be taken to be a sufficiently nice curve. 
To deal with the more general setting of our theorems, we instead rely on the following lemma.   

\begin{lemma}\label{separate}
    Let $E \subseteq U \subseteq \C$, where $E$ is a closed set with locally finite connected components and $U$ is a simply connected open set whose connected components are all unbounded.
    Then there exists an unbounded Arakelian set $\Gamma \subseteq U \setminus E$ such that any connected component of the set $\C \setminus \Gamma$ intersects either $\C \setminus U$ or at most one connected component of $E$. 
\end{lemma}

\begin{proof}
    Since the connected components of $U$ are all unbounded and the connected components of $E$ are locally finite, we may assume that $E$ also only has unbounded connected components.
    For example, we can add a ray, contained in $U \setminus E$, to every bounded component of $E$.
    Note that after this modification, the connected components of $E$ will still be locally finite.

    Next, since the connected components of $E$ are locally finite and pairwise disjoint, 
    we can find a family $\set{V_j}_{j \in J}$ of pairwise disjoint open sets contained in $U$, such that each set $V_j$ contains one connected component of $E$.
    We denote $V = \bigcup_{j \in J} V_j$.
    Since $E$ and $\C \setminus V$ are disjoint closed sets, Urysohn's lemma gives a continuous function $\chi \colon \C \to [0,1]$ such that $\chi(\C \setminus V) = \set{0}$ and $\chi(E) = \set{1}$.
    Pick some $t \in (0,1)$ and note that $\chi^{-1}(\set{t}) \subseteq V \setminus E \subseteq U \setminus E$ is unbounded.
    Since the set $\widehat{\C} \setminus (U \setminus E) = (\widehat{\C} \setminus U) \cup E$ is connected, the set $U \setminus E$ is simply connected. Here, we used the assumption that all the components of $E$ are unbounded.
    Thus Proposition~\ref{Arakelian_so_gosti} tells us that there exists an Arakelian set $\Gamma$ such that $\chi^{-1}(\set{t}) \Subset \Gamma \subseteq U \setminus E$.
    Note that since the set $\chi^{-1}(\set{t})$ is unbounded, the set $\Gamma$ is also unbounded.
    
    It remains to see that $\Gamma$ has the desired separation property. 
    First, let $\eta$ be a path that connects a point in $E$ to a point in $\C \setminus U$. 
    Since $\C \setminus U \subseteq \C \setminus V$, the function $\chi$ achieves the values $0$ and $1$ in the endpoints of $\eta$.
    Then, by the continuity of $\chi$, there exits a point $p \in \eta$ such that $\chi(p) = t$, hence $\eta \cap \Gamma \neq \emptyset$.
    Now suppose that $\eta$ is a path with endpoints in different connected components of $E$. Then these endpoints lie in two disjoint sets $V_j$, and we must have $\eta \cap (\C \setminus V) \neq \emptyset$.
    Thus $\chi$ again achieves the values $0$ and $1$ along $\eta$, so by the same reasoning as before, $\eta \cap \Gamma \neq \emptyset$ follows.
    Both observations together imply that a connected component of the set $\C \setminus \Gamma$ either intersects $\C \setminus U$ or at most one connected component of the set $E$, as desired.
\end{proof}

To construct escaping wandering domains, we employ the same strategy as in the proof of Theorem~\ref{hornWD}, namely,
we use discs to allow the closed set to escapes to infinity uniformly under iteration, and unbounded curves to separate the resulting Fatou components. 
That such sets can always be chosen, and how they are used to construct escaping wandering domains, are established in the following two lemmas.

\begin{lemma}\label{curves_and_discs}
    Let $W \subseteq \C$ be an unbounded connected open set. 
    Then there exist a family of closed discs $(D_n)_{n \ge 1}$ and a family of Arakelian sets $(\gamma_n)_{n \ge 1}$ such that the following holds:
    \begin{enumerate}
        \item The sets $(\gamma_n)_{n \ge 1}$ and $(D_n)_{n \ge 1}$ are pairwise disjoint and locally finite,
        \item $\bigcup_{n \ge 1} (D_n \cup \gamma_n) \subseteq W$ and $D_n \cap D(0,n) = \emptyset$ for all $n \ge 1$,
        \item each disc $D_n$ is contained in a different connected component of the set $\C \setminus \bigcup_{k \ge 1} \gamma_k$. 
    \end{enumerate}
\end{lemma}

\begin{proof}
    Since $W$ is an unbounded connected open set, there exists an unbounded, connected, and simply connected open subset $W_0 \subseteq W$ which has a Jordan curve as its boundary.
    One way to see this is to take a proper embedded unbounded smooth curve in $W$ and then apply the tubular neighbourhood theorem, see~\cite[Theorem 6.24]{Lee}.
    Let $S = (0, \infty) \times (-1,1)$ and let $\phi \colon W_0 \to S$ be a Riemann map, such that $\phi(\infty) = \infty$.
    Note that by our assumptions, the function $\phi$ extends to a homeomorphism on $\partial W_0$.
    Finally, we define a family of curves $(\tilde{\gamma}_n)_{n \ge 1} \subseteq S$ as
    \[\tilde{\gamma}_n = \set{n} \times \left[-2^{-n}, 2^{-n}\right] \cup \set{-2^{-n}, 2^{-n}} \times [n, \infty).\]
    It is easy to see that the family $(\tilde{\gamma}_n)_{n \ge 1}$ consists of locally finite pairwise disjoint Arakelian sets.

    We define the family of curves $(\gamma_n)_{n \ge 1} \subseteq W_0$ by setting $\gamma_n = \phi^{-1}(\tilde{\gamma}_n)$.
    Since $\phi \colon W_0 \to S$ is a homeomorphism and $\phi(\infty) = \infty$, 
    it follows immediately that the family $(\gamma_n)_{n \ge 1}$ consists of locally finite pairwise disjoint closed subsets of $\C$.
    We claim that these sets are also Arakelian. Let $V \subseteq \C$ be an open neighbourhood of $\gamma_n$.
    Then $\phi(W_0 \cap V) \subseteq S$ is an open neighbourhood of the Arakelian set $\tilde{\gamma}_n$, so 
    by Theorem~\ref{characterization} there exists a simply connected open set $U$ such that $\tilde{\gamma}_n \subseteq U \subseteq \phi(W_0 \cap V)$.
    Since $\phi$ is a homeomorphism, $\phi^{-1}(U) \subseteq W_0$ is a simply connected open set such that $\gamma_n \subseteq \phi^{-1}(U) \subseteq V$.
    It now follows from Theorem~\ref{characterization} that the set $\gamma_n$ is Arakelian.

    Finally, since the function $\phi$ extends to a homeomorphism on $\partial W_0$, it is easy to verify that the set $W_0 \setminus \bigcup_{k \ge 1} \gamma_k$ consists of infinitely many unbounded connected components, 
    hence the same holds for the set $\C \setminus \bigcup_{k \ge 1} \gamma_k$.
    We now choose a locally finite family of closed discs $(D_n)_{n \ge 1} \subseteq W$, 
    such that $D_n \cap D(0,n) = \emptyset$ and each disc $D_n$ lies in a different connected component of $\C \setminus \bigcup_{k \ge 1} \gamma_k$.
    Since both families $(\gamma_n)_{n \ge 1}$ and $(D_n)_{n \ge 1}$ are locally finite, their union is also locally finite. 
    Furthermore, it is clear from the construction that all the sets are pairwise disjoint.
\end{proof}

\begin{lemma}\label{modelEscaping}
    Let $W \subseteq \C$ be an unbounded connected open set. 
    Then there exist an Arakelian set $A \subseteq W$, a holomorphic function $h \in \mathcal{O}(A)$, a locally constant function $\varepsilon \colon A \to (0, \infty)$, and a closed disc $B \subseteq A$ such that 
    if $f \in \mathcal{O}(\C)$ is an entire function satisfying $\abs{f(z)-h(z)} < \varepsilon(z)$ for all $z \in A$, then $f$ has an escaping wandering domain containing $B$.
\end{lemma}

\begin{proof}
    Let $(D_n)_{n \ge 1}$ and $(\gamma_n)_{n \ge 1}$ be the sets obtained by applying Lemma~\ref{curves_and_discs} to the open set $W$.
    For $n \ge 1$, let $a_n$ denote the center and $r_n$ the radius of the disc $D_n$. We then define the set 
    \[A = \bigcup_{k=1}^\infty \gamma_k \cup \bigcup_{k=1}^\infty D_k\]
    and choose a closed disc $B \subseteq \mathring{D}_2$. Notice that $B \subseteq A \subseteq W$. 
    Moreover, since the families $(\gamma_n)_{n \ge 1}$ and $(D_n)_{n \ge 1}$ are pairwise disjoint, locally finite, and consist of Arakelian sets, 
    it follows from proposition~\ref{locallyfiniteArakelian} that the set $A$ is also Arakelian.
    Finally, we define the locally constant functions $h \in \mathcal{O}(A)$ and $\varepsilon \colon A \to (0, \infty)$ by
    \[h(z) = \begin{cases} 
        a_1; & z \in \gamma_n \text{ for } n \ge 1 \\ 
        a_1; & z \in D_1\\ 
        a_{n+1}; & z \in D_n \text{ for } n \ge 2 
    \end{cases} 
    \quad \text{and} \quad 
    \varepsilon(z) = \begin{cases}
        \frac{r_1}{2}; & z \in \gamma_n \text{ for } n \ge 1 \\ 
        \frac{r_1}{2}; & z \in D_1\\ 
        \frac{r_{n+1}}{2}; & z \in D_n \text{ for } n \ge 2 
    \end{cases}.\]

    Now suppose that $f \in \mathcal{O}(\C)$ is an entire function satisfying $\abs{f(z)-h(z)} < \varepsilon(z)$ for all $z \in A$.
    It follows that $f(D_1) \subseteq \mathring{D}_1$ as well as $f(\gamma_n) \subseteq \mathring{D}_1$ and $f(D_{n+1}) \subseteq \mathring{D}_{n+2}$ for all $n \ge 1$.
    The function $f$ thus acts as a contraction on the set $D_1$, so $D_1$ is contained in some attracting basin $\mathcal{A}$ and $\bigcup_{k \ge 1} \gamma_k \subseteq f^{-1}(\mathcal{A})$. 
    On the other hand, since $f(D_n) \subseteq \mathring{D}_{n+1}$ and $D_n \cap D(0,n) = \emptyset$ for all $n \ge 2$, 
    the points in the discs $D_n$ escape uniformly to infinity under iteration.
    Thus, the interior of each disc is contained in some Fatou component.
    But since $\bigcup_{k \ge 1} \gamma_k \subseteq f^{-1}(\mathcal{A})$ and each disc $D_n$ is contained in a distinct connected component of the set $\C \setminus \bigcup_{k \ge 1} \gamma_{k}$,
    the same reasoning as in the proof of Theorem~\ref{hornWD} shows that these Fatou components are disjoint and hence form the orbit of an escaping wandering domain containing $B$.
\end{proof}

For the oscillating case, we instead make use of a classical result by Eremenko and Lyubich. 
In~\cite[Example 1]{EremenkoLjubichPathologicalExamples}, the authors use Arakelian approximation to construct an entire function with an oscillating wandering domain. 
This is, in fact, the first known example of an entire function exhibiting an oscillating wandering domain. 
A careful examination of the proof reveals that the construction is parametric in the following sense.

\begin{lemma}\label{modelOscillating}
    Let $(r_m)_{m \ge 1}$ be a decreasing sequence of positive numbers and $(a_m)_{m \ge 1}$ a sequence of complex numbers such that $\lvert a_{m+1} \rvert > 2r_1 + \lvert a_m \rvert$ for all $m \ge 1$.
    Then there exist an Arakelian set $A \subseteq \C$, a holomorphic function $h \in \mathcal{O}(A)$, a locally constant function $\varepsilon \colon A \to (0, \infty)$, and a closed disc $B \subseteq A$ such that:
    \begin{enumerate}
        \item $\{a_m \mid m \ge 1\} \subseteq A \subseteq \bigcup_{m \ge 1} D(a_m, r_m)$,
        \item if $f \in \mathcal{O}(\C)$ is an entire function satisfying $\abs{f(z)-h(z)} < \varepsilon(z)$ for all $z \in A$, as well as $f(a_m) = h(a_m)$ and $f'(a_m) = h'(a_m)$ for all $m \ge 1$, then $f$ has an oscillating wandering domain containing $B$.
    \end{enumerate}
\end{lemma}

We are now ready to prove Theorems~\ref{approxWD_escaping} and~\ref{approxWD_oscillating}.

\begin{proof}[Proof of Theorem~\ref{approxWD_escaping}]
    We begin by applying Proposition~\ref{Arakelian_so_gosti} to the closed set $F$ and the simply connected open set $U$ to obtain an Arakelian set $F \Subset E \subseteq U$ with locally finite connected components.
    Next, we apply Lemma~\ref{separate} to the sets $E$ and $U$ to obtain an unbounded Arakelian set 
    $\Gamma \subseteq U \setminus E$ such that any connected component of the set $\C \setminus \Gamma$ intersects either $\C \setminus U$ or at most one connected component of $E$.
    Notice that since the set $\C \setminus U$ is non-empty, there exists a connected component $W$ of $\C \setminus \Gamma$ that intersects $\C \setminus U$. 
    The set $W$ is open and unbounded, the latter following from $\Gamma$ being an Arakelian set, 
    so we may apply Lemma~\ref{modelEscaping} to obtain an Arakelian set $A \subseteq \C$, a holomorphic function $h_0 \in \mathcal{O}(A)$, a locally constant function $\varepsilon_0 \colon A \to (0, \infty)$, and a closed disc $B \subseteq A$ as described in the lemma.
    Let $B = D(b, \rho)$, where $b$ and $\rho$ denote the center and radius of $B$ respectively, and finally choose a closed disc  $D = D(a, r) \subseteq W \setminus A$.
    
    We now define the set
    \[\Pi = E \cup \Gamma \cup A \cup D,\]
    see Figure~\ref{approxWDescaping_fig}. 
    Since the sets $E$, $\Gamma$, $A$, and $D$ are pairwise disjoint, locally finite, and Arakelian, 
    it follows from Proposition~\ref{locallyfiniteArakelian} that the set $\Pi$ is also Arakelian. 
    Next we define the locally constant functions $h \in \mathcal{O}(\Pi)$ and $\varepsilon \colon \Pi \to (0,\infty)$ by
    \[h(z) = \begin{cases} 
        b; & z \in E \\
        a; & z \in \Gamma \\
        h_0(z); & z \in A \\
        a; & z \in D\\ 
    \end{cases} 
    \quad \text{and} \quad 
    \varepsilon(z) = \begin{cases}
        \frac{\rho}{2}; & z \in E \\
        \frac{r}{2}; & z \in \Gamma \\
        \varepsilon_0(z); & z \in A \\
        \frac{r}{2}; & z \in D\\ 
    \end{cases}.\]
    Applying Proposition~\ref{PosplositevArakelian} we obtain an entire function $f \in \mathcal{O}(\C)$ such that 
    $\abs{f(z)-h(z)} < \varepsilon(z)$ for all $z \in \Pi$.
    In particular, this implies 
    $f(D) \subseteq \mathring{D}$, $f(\Gamma) \subseteq \mathring{D}$, $f(E) \subseteq \mathring{B}$ and that $f$ has an escaping wandering domain containing $B$.
    
    The function $f$ acts as a contraction on $D$, so the set is contained in some attracting basin $\mathcal{A}$, and it follows that $\Gamma \subseteq f^{-1}(\mathcal{A})$. 
    Now let $F_0$ be a connected component of the set $F$. 
    Since $F_0 \Subset E$ and $f(E) \subseteq \mathring{B}$, the set $F_0$ is contained is some escaping wandering domain $\Omega$.
    Furthermore, since the unbounded set $\Gamma$ is contained in some preimage of the attracting basin $\mathcal{A}$, 
    the function $f$ has an unbounded Fatou component. It then follows from~\cite[Theorem 3.1.6]{HolomorphicDynamics} that the wandering domain $\Omega$ is simply connected.
    It remains to verify that $\Omega \subseteq U$.   
    If this is not the case, there exists a curve $\eta \subseteq \Omega$ connecting a point in $E$ to a point in $\C \setminus U$.
    Since by construction the endpoints of this curve lie in different components of the set $\C \setminus \Gamma$,
    the curve $\eta$ must intersect $\Gamma$ and so contains a point that gets sent to $\mathcal{A}$. 
    This, however, contradicts the fact that all points in $\Omega$ escape uniformly to infinity under iteration. 
\end{proof}

\begin{figure}[ht]
    \centering
    \includegraphics[width=0.6\textwidth]{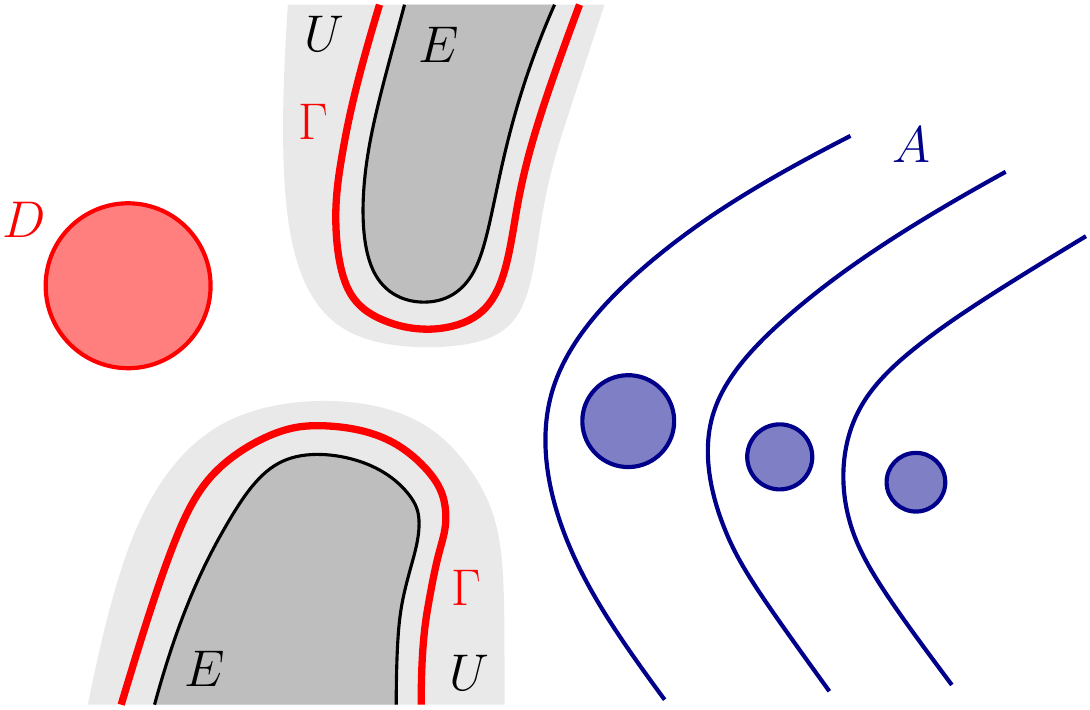}
    \caption{The construction of the set $\Pi$ in the proof of Theorem~\ref{approxWD_escaping}. The set $A$ allows us to construct escaping wandering domains, while the sets $D$ and $\Gamma$ ensure that these wandering domains are contained in $U$.}
    \label{approxWDescaping_fig}  
\end{figure}

\begin{proof}[Proof of Theorem~\ref{approxWD_oscillating}]
    The proof follows the same strategy as in the case of escaping wandering domains, 
    except that we use Lemma~\ref{modelOscillating}, which provides a model for oscillating wandering domains, 
    instead of Lemma~\ref{modelEscaping}, which provides a model for escaping ones.

    We again begin by applying Proposition~\ref{Arakelian_so_gosti} to the closed set $F$ and the simply connected open set $U$ to obtain an Arakelian set $F \Subset E \subseteq U$ with locally finite connected components.
    Next, we apply Lemma~\ref{separate} to the sets $E$ and $U$ to obtain an unbounded Arakelian set 
    $\Gamma \subseteq U \setminus E$, such that any connected component of the set $\C \setminus \Gamma$ intersects either $\C \setminus U$ or at most one connected component of $E$.
    Notice that since the set $\C \setminus U$ is non-empty, there exists a connected component $W$ of $\C \setminus \Gamma$ that intersects $\C \setminus U$. 
    The set $W$ is open and unbounded, so we can choose a decreasing sequence of positive numbers $(r_m)_{m \ge 1}$ and a sequence of complex numbers $(a_m)_{m \ge 1}$ such that $\abs{a_{m+1}}> 2r_1 + \abs{a_m}$ and $D(a_m, r_m) \subseteq W$ for all $m \ge 1$.
    We then apply Lemma~\ref{modelOscillating} to these two sequences and obtain an Arakelian set $A \subseteq W$, a holomorphic function $h_0 \in \mathcal{O}(A)$, a locally constant function $\varepsilon_0 \colon A \to (0, \infty)$, and a closed disc $B \subseteq A$ as described in the lemma.
    Let $B = D(b, \rho)$, where $b$ and $\rho$ denote the center and radius of $B$ respectively, and finally choose a closed disc  $D = D(a, r) \subseteq W \setminus A$.
    
    We now define the set
    \[\Pi = E \cup \Gamma \cup A \cup D,\]
    see Figure~\ref{approxWDoscillating_fig}. 
    Since the sets $E$, $\Gamma$, $A$, and $D$ are pairwise disjoint, locally finite, and Arakelian, 
    it follows from Proposition~\ref{locallyfiniteArakelian} that the set $\Pi$ is also Arakelian. 
    Next we define a holomorphic function $h \in \mathcal{O}(\Pi)$ and a locally constant function $\varepsilon \colon \Pi \to (0,\infty)$ by 
    \[h(z) = \begin{cases} 
        b; & z \in E \\
        a; & z \in \Gamma \\
        h_0(z); & z \in A \\
        a; & z \in D\\ 
    \end{cases} 
    \quad \text{and} \quad 
    \varepsilon(z) = \begin{cases}
        \frac{\rho}{2}; & z \in E \\
        \frac{r}{2}; & z \in \Gamma \\
        \varepsilon_0(z); & z \in A \\
        \frac{r}{2}; & z \in D\\ 
    \end{cases}.\]
    Applying Proposition~\ref{PosplositevArakelian}, we obtain an entire function $f \in \mathcal{O}(\C)$ such that 
    $\abs{f(z)-h(z)} < \varepsilon(z)$ for all $z \in \Pi$, as well as $f(a_m) = h(a_m)$ and $f'(a_m) = h'(a_m)$ for all $m \ge 1$.
    The latter can be ensured, since each point $a_m$ is contained in a distinct connected component of the Arakelian set $A$.
    It then follows from the construction that $f(D) \subseteq \mathring{D}$, $f(\Gamma) \subseteq \mathring{D}$, $f(E) \subseteq \mathring{B}$ and that $f$ has an oscillating wandering domain containing $B$.

    The function $f$ acts as a contraction on $D$, so the set is contained in some attracting basin $\mathcal{A}$, and it follows that $\Gamma \subseteq f^{-1}(\mathcal{A})$. 
    Now let $F_0$ be a connected component of the set $F$.
    Since $F_0 \Subset E$ and $f(E) \subseteq \mathring{B}$, the set $F_0$ is contained in some oscillating wandering domain $\Omega$.
    The same arguments as at the end of the proof of Theorem~\ref{approxWD_escaping} then show that $\Omega$ is simply connected and contained in $U$.
\end{proof}

\begin{figure}[ht]
    \centering
    \includegraphics[width=0.6\textwidth]{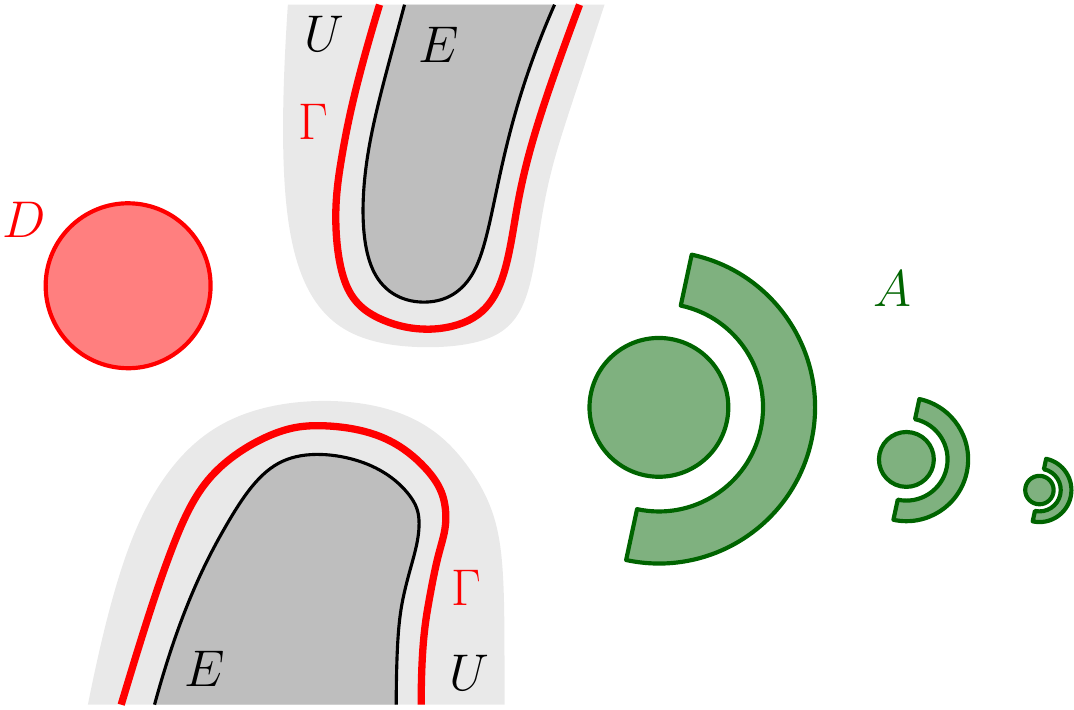}
    \caption{The construction of the set $\Pi$ in the proof of Theorem~\ref{approxWD_oscillating}. The set $A$ allows us to construct oscillating wandering domains, while the sets $D$ and $\Gamma$ ensure that these wandering domains are contained in $U$.}
    \label{approxWDoscillating_fig}  
\end{figure}

\begin{remark}
    Note that in the above constructions, each connected component of $E$ is contained in a distinct wandering domain of the function $f$.
    This follows from the same argument used to prove that the wandering domains are contained in $U$.
    However, after one iteration, all these wandering domains are sent to the same Fatou component.
    One can modify the construction by applying either Lemma~\ref{modelEscaping} or Lemma~\ref{modelOscillating} once
    for each connected component of $E$, and then sending each component to its own Arakelian set given by the lemmas.
    This way, the orbits of wandering domains containing different connected components of $E$ remain disjoint under iteration.
    Moreover, each component can be chosen independently to lie in either an escaping or an oscillating wandering domain. 
    See the proof of Theorem~\ref{combination} for more details.
\end{remark}

\begin{remark}
    If the connected components of $F$ are not locally finite, they cannot be contained in distinct wandering domains.
    If, however, $F$ does have locally finite connected components, this can be achieved.
    To do so, let $\set{V_j}_{j \in J}$ be a locally finite family of pairwise disjoint open sets such that each set $V_j$ contains one connected component of $F$.
    Since a locally finite intersection of simply connected open sets is simply connected,
    we may apply Proposition~\ref{Arakelian_so_gosti} to the closed set $F$ and the simply connected open set $U \cap \bigcup_{j \in J} V_j$
    to obtain the Arakelian set $E$. The proof then proceeds as above.
\end{remark}

The remainder of this section is devoted to proving Theorem~\ref{combination}. 
Before presenting the proof, we state a final lemma that allows us to prescribe the geometry of bounded wandering domains.

\begin{lemma}\label{PreciseShape}
    Let $K \subseteq W \subseteq \C$, where $K$ is a compact set without holes and $W$ is an unbounded open set.
    Then there exists an Arakelian set $K \subseteq A \subseteq W$, a holomorphic function $h \in \mathcal{O}(A)$, and a locally constant function $\varepsilon \colon A \to (0, \infty)$ such that the following holds.
    If $f \in \mathcal{O}(\C)$ is an entire function satisfying $\abs{f(z)-h(z)} < \varepsilon(z)$ for all $z \in A$, then every connected component of $\mathring{K}$ is a simply connected wandering domains of $f$.
    These wandering domains can be chosen to be either escaping or oscillating.
\end{lemma}

\begin{proof}
    The lemma follows by modifying the construction in the proof of~\cite[Theorem 1]{GeometryOfWD} in the escaping case 
    and in the proof of~\cite[Theorem 2]{GeometryOfWD} in the oscillating case. 
    Here we briefly outline the required modifications and leave the remaining details to the reader.

    First, we clarify the change in perspective. 
    The theorems in~\cite{GeometryOfWD} state that a bounded connected regular open set $\Omega$, whose closure has connected complement, is a wandering domain of some entire function. 
    However, the same proof shows that, given a compact set $K$ without holes, there exists an entire function for which every connected component of $\mathring{K}$ is a wandering domain. 
    Instead of starting with an open set $\Omega$ and taking its closure, we begin with the compact set $K$ and consider its interior.
    This is also the approach taken by the authors in~\cite{EremenkoConj}.

    Next, observe that the proofs of the above results rely on a Runge approximation theorem with interpolation.
    This is not necessary, and one can instead use standard Runge approximation.
    Indeed, the images of the interpolation points do not need to be prescribed precisely; they must only be sent to an attracting basin, which is an open set.
    This generalization is also demonstrated in the proof of~\cite[Theorem 3.1]{EremenkoConj}.
    See also the proofs of Theorems~\ref{stripWDescaping} and Theorems~\ref{stripWDoscillating} in Section~\ref{sec3}.

    Suppose that $K \subseteq D(3,1)$, and let $f$ be the entire function given by one of the theorems in~\cite{GeometryOfWD}.
    The function $f$ is obtained as the limit of a sequence of entire functions $(f_n)_{n \ge 0}$ constructed inductively using Runge approximation.
    For the base case, we set $f_0(z) = z/2$.
    For $n \ge 1$, assuming that $f_{n-1}$ has already been constructed, 
    we define a compact set without holes $A_n \subseteq D(4n -1, 1)$, a holomorphic function $h_n \in \mathcal{O}(D(0, 4n-3) \cup A_n)$ and a positive number $\varepsilon_n > 0$.
    The function $h_n$ agrees with $f_{n-1}$ on $D(0, 4n-3)$ and is piecewise linear on $A_n$.
    We then obtain $f_n$ by approximating $h_n$ on the set $D(0,4n-3) \cup A_n$ with an error of at most $\varepsilon_n$.
    Approximating on the sets $D(0,4n-3)$ ensures that the sequence $(f_n)_{n \ge 1}$ converges uniformly on compact sets to the function $f$,
    while approximating on the sets $A_n$ and $D(0,1)$ guarantees the desired dynamical properties of $f$.
    In particular, the latter follows since the function $f$ satisfies $\norm{f - h_1}_{D(0,1)} < \varepsilon_1$ and $\norm{f - h_n}_{A_n} < \varepsilon_n$ for all $n \ge 1$.
    We now define the Arakelian set
    \[A = D(0,1) \cup \bigcup_{k \ge 1} A_k,\]
    together with a holomorphic function $h \in \mathcal{O}(A)$ given by
    \[h(z) = \begin{cases} 
        h_1(z); & z \in D(0,1) \\
        h_k(z); & z \in A_k \\
    \end{cases},
    \]
    and a locally constant function $\varepsilon \colon A \to (0, \infty)$ such that
    \[\varepsilon \vert_{D(0,1)} < \varepsilon_1 - \norm{f - h_1}_{D(0,1)} \quad \text{and} \quad \varepsilon \vert_{A_k} < \varepsilon_k - \norm{f - h_k}_{A_k}\]
    for all $k \ge 1$. 
    Then, given an entire function $g \in \mathcal{O}(\C)$ satisfying $\abs{g(z)-h(z)} < \varepsilon(z)$ for all $z \in A$,
    we have $\norm{g - h_1}_{D(0,1)} < \varepsilon_1$ and $\norm{g - h_n}_{A_n} < \varepsilon_n$ for all $n \ge 1$.
    It follows that the connected components of $\mathring{K}$ are also wandering domains of the function $g$.
    Since the set $K$ has no holes, these wandering domains are simply connected.

    Finally, we address the condition $K \subseteq A \subseteq W$.
    Let $(D_n)_{n \ge 1}$ be a family of pairwise disjoint closed discs satisfying $D_n \subseteq W \setminus K$ and $D_n \cap D(0,n) = \emptyset$ for all $n \ge 1$.
    We also choose a compact set $H$ without holes such that $K \Subset H \subseteq W \setminus \bigcup_{m \ge 1}{D_m}$.
    By modifying the linear maps used in the definitions of the functions $(h_n)_{n \ge 1}$,
    we can replace the sets $D(0,1)$, $D(3,1)$, and $D(4n -1, 1)$ by the sets $D_1$, $H$, and $D_n$, respectively, for $n \ge 2$.
    Then $K \subseteq A \subseteq W$, and since the family $(D_n)_{n \ge 1}$ is locally finite, the remaining arguments of the proof remain valid.
\end{proof}

\begin{proof}[Proof of Theorem~\ref{combination}]
    Without loss of generality we may assume that the connected components of $U$ are unbounded. 
    Indeed, since $U$ is simply connected, the connected components of $\C \setminus U$ are unbounded, 
    so any bounded component of $U$ can be extended to be unbounded and still remain simply connected.
    
    We begin by applying Lemma~\ref{separate} to the sets $E$ and $U$ to obtain an unbounded Arakelian set 
    $\Gamma \subseteq U \setminus E$, such that any connected component of the set $\C \setminus \Gamma$ intersects either $\C \setminus U$ or at most one connected component of $E$.
    Let $(E_n)_{n \ge 1}$ be the connected components of $E$, there are only countably many since they are locally finite,
    and let $V_n$ denote the connected component of $\C \setminus \Gamma$ that contains $E_n$.
    Since $\C \setminus U \neq \emptyset$, there also exists a connected component $W$ of $\C \setminus \Gamma$ that intersects $\C \setminus U$.
    Note that the sets $W$ and $(V_n)_{n \ge 1}$ are unbounded.
    In particular, we can choose a pairwise disjoint locally finite family $(W_n)_{n \ge 1}$ of connected unbounded open subsets of $W$. 
    Finally, we choose a closed disc $D = D(a,r) \subseteq W$ disjoint from all the sets $W_n$.

    Depending on whether the set $E_n$ is bounded or unbounded, we proceed in one of two ways:
    \begin{itemize}
        \item If $E_n$ is bounded, then, being Arakelian, it is a compact set without holes.
        Thus, we can apply Lemma~\ref{PreciseShape} to the sets $E_n$ and $V_n$. 
        After choosing whether the associated wandering domains should be escaping or oscillating, 
        we obtain an Arakelian set $A_n$, a holomorphic function $h_n \in \mathcal{O}(A_n)$, and a locally constant function $\varepsilon_n \colon A_n \to (0, \infty)$.

        \item If $E_n$ is unbounded, we instead use Lemma~\ref{modelEscaping} or Lemma~\ref{modelOscillating}, according to whether we want the associated wandering domain to be escaping or oscillating.
        In the escaping case, we apply Lemma~\ref{modelEscaping} to the sets $E_n$ and $W_n$.
        In the oscillating case, we choose appropriate sequences $(r_m^n)_{m \ge 1}$ and $(a_m^n)_{m \ge 1}$ such that $D(a_m^n, r_m^n) \subseteq W_n$ for all $m \ge 1$ and then apply Lemma~\ref{modelOscillating}.
        Thus we obtain an Arakelian set $A_n \subseteq W_n$, a holomorphic function $h_n \in \mathcal{O}(A_n)$, a locally constant function $\varepsilon_n \colon A_n \to (0, \infty)$, and a closed disc $B_n \subseteq A_n$.
    \end{itemize}
    
    \begin{figure}[t]
    \centering
    \includegraphics[width=0.545\textwidth]{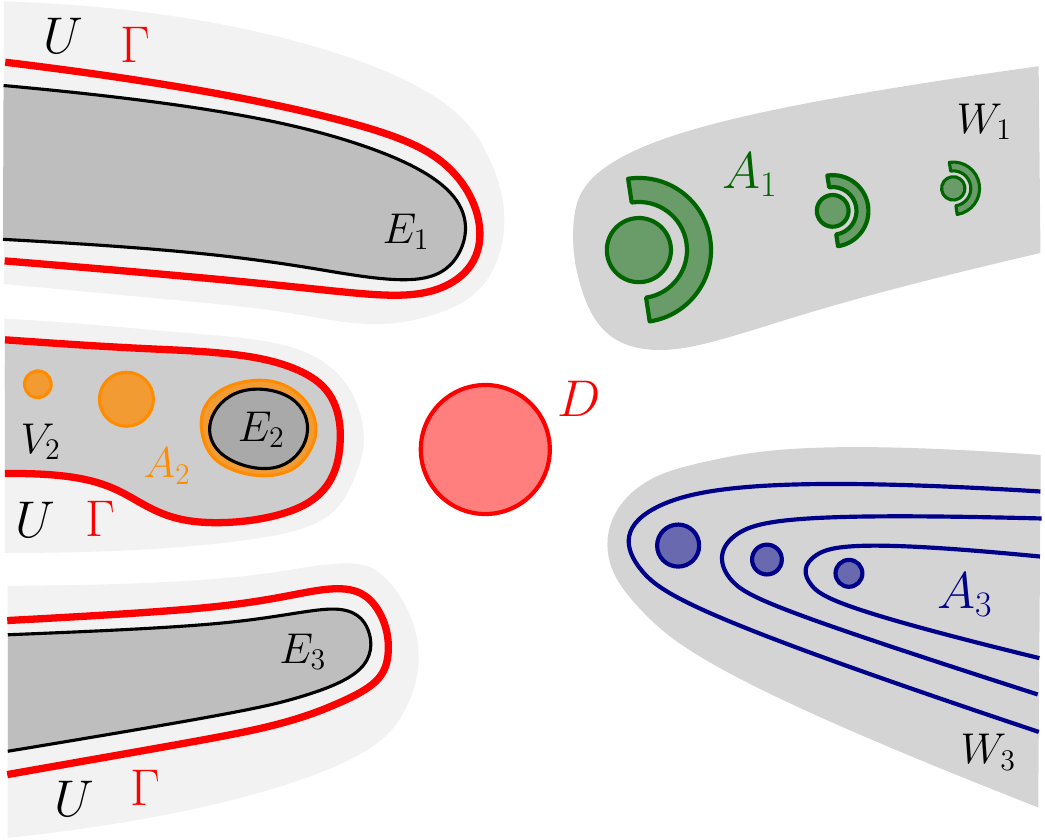}
    \caption{The construction of the set $\Pi$ in the proof of Theorem~\ref{combination}. 
    Each connected component of $E$ has an associated Arakelian set that facilitates the desired dynamics.
    In this example, $E_1$ and $E_3$ will be in oscillating and escaping wandering domain, respective, while $\mathring{E}_2$ will be an escaping wandering domain.}
    \label{approxWDcombination_fig}  
    \end{figure}

    We now define the set 
    \[\Pi = E \cup \Gamma \cup \bigcup_{k \ge 1} A_k \cup D,\]
    see Figure~\ref{approxWDcombination_fig}.
    If $E_n$ is bounded, then $A_n$ is contained in its own connected component of $\C \setminus \Gamma$.
    If $E_n$ is unbounded, then $A_n \subseteq W_n$, where $(W_n)_{n \ge 1}$ is a locally finite family of subsets of $W$.
    Thus the collection of sets appearing in the above union is locally finite. 
    Since these sets are also pairwise disjoint and Arakelian, Proposition~\ref{locallyfiniteArakelian} implies that $\Pi$ is an Arakelian set.
    Next we define the holomorphic function $h \in \mathcal{O}(\Pi)$ and the locally constant function $\varepsilon \colon \Pi \to (0, \infty)$ by 
    \[h(z) = \begin{cases} 
        b_n; & z \in E_n \text{ unbounded}\\
        a; & z \in \Gamma \\
        h_n(z); & z \in A_n \\
        a; & z \in D\\ 
    \end{cases}
    \quad \text{and} \quad 
    \varepsilon(z) = \begin{cases}
        \frac{\rho_n}{2}; & z \in E_n \text{ unbounded}\\
        \frac{r}{2}; & z \in \Gamma \\
        \varepsilon_n(z); & z \in A_n \\
        \frac{r}{2}; & z \in D\\ 
    \end{cases}.\]
    Note that both functions are indeed defined on the entire set $\Pi$, since $E_n \subseteq A_n$ whenever $E_n$ is bounded.
    We now apply Proposition~\ref{PosplositevArakelian} to obtain an entire function $f \in \mathcal{O}(\C)$ such that 
    $\abs{f(z)-h(z)} < \varepsilon(z)$ for all $z \in \Pi$, as well as $f(a_m^n) = h(a_m^n)$ and $f'(a_m^n) = h'(a_m^n)$ for all $m \ge 1$ and $n \ge 1$ with $E_n$ bounded.
    This is possible since each point $a_m^n$ is contained in a distinct connected component of the Arakelian set $A_n$.
    It then follows, by the same arguments as in the previous proofs in this section, that $f$ has the desired properties.
\end{proof}
    
\printbibliography

\end{document}